\documentclass[11pt]{amsart}
\usepackage{amsmath,amssymb,amsthm,mathrsfs}
\usepackage[margin=1.15in]{geometry}
\usepackage{enumitem}
\usepackage[hidelinks]{hyperref}

\newtheorem{theorem}{Theorem}[section]
\newtheorem{proposition}[theorem]{Proposition}
\newtheorem{lemma}[theorem]{Lemma}
\newtheorem{corollary}[theorem]{Corollary}
\newtheorem{conjecture}[theorem]{Conjecture}
\theoremstyle{definition}
\newtheorem{definition}[theorem]{Definition}

\theoremstyle{remark}

\numberwithin{equation}{section}

\makeatletter
\renewcommand\section{\@startsection{section}{1}{\z@}%
  {.7\linespacing\@plus\linespacing}{.5\linespacing}%
  {\normalfont\large\bfseries\raggedright}}
\renewcommand\l@section{\@tocline{1}{0pt}{1pc}{}{\bfseries}}
\makeatother

\newcommand{\La}{\Lambda}
\newcommand{\Ld}{\Lambda^{!}}
\newcommand{\inj}{\text{-}\mathrm{inj}^{\mathbb Z}}
\newcommand{\Proj}{\text{-}\mathrm{Proj}^{\mathbb Z}}
\newcommand{\gmod}{\text{-}\mathrm{gmod}}
\newcommand{\GMod}{\text{-}\mathrm{GMod}}
\newcommand{\Cop}
{\text{-}\mathrm{Cop}^{\mathbb Z}}
\newcommand{\fmod}{\text{-}\mathrm{mod}}
\newcommand{\Fcp}{\text{-}\mathrm{Fcp}^{\mathbb Z}}
\newcommand{\Fp}{\text{-}\mathrm{Fp}^{\mathbb Z}}

\newcommand{\Inj}{\text{-}\mathrm{Inj}^{\mathbb Z}}
\newcommand{\proj}{\text{-}\mathrm{proj}^{\mathbb Z}}
\newcommand{\Gproj}{\text{-}\underline{\mathrm{Gproj}}}
\newcommand{\Db}{\mathsf{D}^{b}}
\newcommand{\Dsg}{\mathsf{D}_{\mathrm{sg}}}
\newcommand{\Kb}{\mathsf{K}}
\newcommand{\C}{\mathsf{C}}
\newcommand{\Kd}{\mathfrak{K}}
\newcommand{\Fc}{\mathcal{F}}
\newcommand{\FF}{\mathfrak{F}}
\newcommand{\GG}{\mathfrak{G}}
\newcommand{\TT}{\mathfrak{T}}
\newcommand{\LC}{\mathcal{LC}}
\newcommand{\Tot}{\operatorname{Tot}}
\newcommand{\Lin}{\operatorname{Lin}}
\newcommand{\ld}{\operatorname{ld}}
\newcommand{\glld}{\operatorname{gl\,ld}}
\newcommand{\soc}{\operatorname{soc}}
\newcommand{\rad}{\operatorname{rad}}
\newcommand{\Hom}{\operatorname{Hom}}
\newcommand{\uHom}{\operatorname{\underline{Hom}}}

\newcommand{\Ext}{\operatorname{Ext}}
\newcommand{\Ker}{\operatorname{Ker}}
\newcommand{\Coker}{\operatorname{Coker}}
\newcommand{\Img}{\operatorname{Im}}
\newcommand{\add}{\operatorname{add}}
\newcommand{\pd}{\operatorname{pd}}
\newcommand{\grfindim}{\operatorname{grfin\,dim}}
\newcommand{\findim}{\operatorname{fin\,dim}}
\newcommand{\CHL}{\operatorname{CHL}}
\newcommand{\lgr}{\ell_{\mathrm{gr}}}
\newcommand{\DD}{\mathbb{D}}
\newcommand{\sh}[1]{\langle #1\rangle}
\newcommand{\id}{\mathrm{id}}
\newcommand{\pretr}{\operatorname{pretr}}
\newcommand{\Ia}[1]{\mathcal{I}_{#1}}

\begin{document}

\title{Koszul duality for finite-dimensional Koszul algebras}
\author{A. M. Bouhada}
\date{}

\begin{abstract}
Let $\La$ be a finite-dimensional Koszul algebra with Koszul dual $\Ld$. Using the Koszul
functors and the homological perturbation lemma, we describe the minimal projective resolution of every
object of the bounded derived category of finite-dimensional graded $\La$-modules in terms
of the cohomology of the corresponding complex of coperfect graded $\Ld$-modules. In
particular, the linear part of the minimal graded projective resolution of a
finite-dimensional graded $\La$-module is a direct sum of Koszul complexes of these
cohomology modules. We deduce that $\La$ is absolutely Koszul if and only if $\Ld$ is
graded left co-coherent, if and only if the category of coperfect graded $\Ld$-modules is
abelian. As applications, we show that absolutely Koszul algebras of finite global
linearity defect have finite graded finitistic dimension; we refine
the graded and ungraded Koszul dualities and construct nonstandard $t$-structures on the
bounded derived and singularity categories. Finally,
we apply these results to finite-dimensional quadratic monomial algebras.


\end{abstract}

\maketitle

\tableofcontents

\section{Introduction}

In \cite{Bo1} we established, for a finite-dimensional Koszul algebra $\La$ with Koszul dual
$\Ld$, several triangulated equivalences relating the derived and singularity categories of
$\La$ to bounded derived categories of graded modules over $\Ld$. At the graded level, the two basic ones are
\begin{equation*}
\FF\colon \Db\big(\Ld\Cop\big)\xrightarrow{\ \sim\ }\Db\big(\La\gmod\big),
\qquad
\Db\big(\Ld\Cop/\Ld\gmod\big)\xrightarrow{\ \sim\ }\Dsg\big(\La\gmod\big),
\end{equation*}
where $\Ld\Cop$ denotes the exact category of locally finite-dimensional \emph{coperfect}
graded $\Ld$-modules, namely those admitting a finite injective coresolution by finitely
cogenerated graded injective modules, and $\Ld\gmod$ is its Serre subcategory of
finite-dimensional modules. The equivalence $\FF$ is induced by the Koszul functor, and the
second one is obtained from it by passing to Verdier quotients. No finiteness condition,
such as noetherianity or coherence, is imposed on $\Ld$. Ungraded versions, describing
$\Db(\La\fmod)$ and $\Dsg(\La\fmod)$, follow by means of dg orbit categories \cite{Bo1}.

The purpose of the present paper is to study these equivalences in more depth, and to
understand which homological properties of objects over $\La$ they reflect on the side of
the Koszul dual. Two questions guide us. The first concerns minimal projective resolutions:
how can the minimal projective resolution of an object $Y$ of $\Db(\La\gmod)$, and in
particular its linear part, be read off from the complex $\GG(Y)$ of coperfect graded
$\Ld$-modules, where $\GG$ is a quasi-inverse of $\FF$? The second concerns the category
$\Ld\Cop$ itself, which is in general only an exact category: when is it abelian, and what
does this property say about $\La$?

For the first question, recall that $\FF$ is computed on a bounded complex $X$ of graded
$\Ld$-modules by applying the Koszul functor $K$ termwise, which gives a double complex
$\Kd(X)$, and taking the total complex $\Fc(X)=\Tot\Kd(X)$. The horizontal differential of
$\Kd(X)$ has weight zero, whereas the vertical one has weight one, so that $\Fc(X)$ is in
general far from minimal. Splitting off the weight-zero part and applying the perturbation
lemma, we show that $\Fc(X)$ is homotopy equivalent to a minimal complex whose terms are
obtained from the cohomology modules of $X$, and whose linear part is the direct sum of the
complexes $K(H^p(X))[-p]$, the summand indexed by $p$ being the linear strand of weight $-p$
(Theorem~\ref{thm:structure}). In particular, the linear part of the minimal graded
projective resolution of a module $N\in\La\gmod$ is determined by the cohomology modules of
$\GG(N)$ (Corollary~\ref{cor:linear-part}). For the linear part, this result originates in the work of Eisenbud,
Fl{\o}ystad and Schreyer on exterior algebras \cite[Corollary~3.6]{EFS} and was extended by
Yanagawa \cite[Theorem~3.9]{Ya1}; these results describe the linear part only, whereas
Theorem~\ref{thm:structure} describes the whole minimal complex. In our formulation it is an isomorphism of complexes, it
holds without any finiteness hypothesis on $\Ld$, and the perturbation gives the non-linear
components of the differential explicitly (Theorem~\ref{thm:structure}(d)). We also recover the
Ext-groups and the projective dimension of $N$ from the cohomology of $\GG(N)$
(Corollaries~\ref{cor:ext} and~\ref{cor:pd}), and characterise weakly Koszul modules
(Corollary~\ref{cor:weakly-koszul}).

The second question leads to absolutely Koszul algebras, that is, Koszul algebras over which
every finite-dimensional graded module has finite linearity defect. They were introduced by
Iyengar and R\"omer in the commutative setting \cite{IR}; see also \cite{HI}, and they are
closely related to the linear parts of free resolutions over exterior algebras studied in
\cite{EFS}. The corresponding homological condition had already appeared in the
noncommutative setting in the work of Mart\'inez-Villa and Zacharia \cite{MZ}, although the
terminology was not used there; we adopt it throughout, keeping in mind that absolutely
Koszul algebras in the commutative setting satisfy additional properties which need not
persist in the noncommutative case \cite{CINR}. Yanagawa \cite{Ya1,Ya2} used the Koszul
duality of \cite{BGS} to relate linearity defect and regularity over $\La$ to modules over
the Koszul dual, and proved in particular that absolute Koszulity implies a coherence
property of the Koszul dual (his convention differs from ours by passage to the opposite
algebra). Conversely, Mart\'inez-Villa and Zacharia proved absolute Koszulity under the
substantially stronger assumption that the Koszul dual is graded noetherian \cite{MZ}.

Using the description of linear parts, we show that $\La$ is absolutely Koszul if and only
if every object of $\Db(\La\gmod)$ has finite linearity defect, if and only if $\Ld\Cop$ is
abelian, if and only if $\Ld$ is graded left co-coherent; in this case $\Ld\Cop$ coincides
with the category $\Ld\Fcp$ of finitely copresented graded modules (Theorem~\ref{thm:A}).
Hence noetherianity of the Koszul dual is not needed in the result of Mart\'inez-Villa and
Zacharia. This also sheds light on coherence phenomena for Koszul algebras of finite global
dimension; compare \cite[Conjecture~1.3]{Pi}.

We then consider the graded finitistic dimension of $\La$. The projective dimension of a
module $N$ is the graded cohomological length of $\GG(N)$ minus one
(Theorem~\ref{thm:grfindim}), and if the global linearity defect of $\La$ is finite, then
the graded finitistic dimension is bounded by the global linearity defect plus the
finitistic dimension of the category of Koszul modules, which is finite
(Corollary~\ref{cor:A}).

When $\La$ is absolutely Koszul, the dualities of \cite{Bo1} hold with $\Ld\Cop$ replaced by
the abelian category $\Ld\Fcp$ (Corollary~\ref{cor:B}), and the quotient $\Ld\Fcp/\Ld\gmod$
is an abelian ``cotails'' category, the dual analogue of the tails categories of
noncommutative projective geometry \cite{AZ,Pi,Po}. The standard $t$-structures of
$\Db(\Ld\Fcp)$ and $\Db(\Ld\Fcp/\Ld\gmod)$ can then be transported to $\Db(\La\gmod)$ and
$\Dsg(\La\gmod)$. We describe the resulting $t$-structure on $\Db(\La\gmod)$ intrinsically:
an object lies in the aisle if and only if the $n$-th term of its minimal projective
resolution is generated in degrees $\le-n$ for all $n$, and the heart consists of the objects
whose minimal projective resolution is linear; in particular, the modules in the heart are
precisely the Koszul modules (Theorem~\ref{thm:t-derived}). The heart of the induced
$t$-structure on $\Dsg(\La\gmod)$ is equivalent to $\Ld\Fcp/\Ld\gmod$, and its objects are the
images of shifted linear modules (Theorem~\ref{thm:t-singular}).

Green, Mart\'inez-Villa, Reiten, Solberg and Zacharia \cite{GMRSZ} observed that no
necessary and sufficient condition was known for the category $\mathcal L(\La)$ of modules
with a linear presentation to coincide with the category $\mathcal K(\La)$ of Koszul
modules; several necessary or sufficient conditions were obtained subsequently
\cite{AG,GMZ}. We give such a condition in terms of finite-dimensional graded $\Ld$-modules
concentrated in degrees $-1$ and $0$ (Theorem~\ref{thm:C}).

Finally, we consider finite-dimensional quadratic monomial algebras. They are absolutely
Koszul with global linearity defect at most two (Theorem~\ref{thm:mon-abs-Koszul}), so the
preceding results apply to them. The finitely copresented graded $\Ld$-modules are, up to
finite-dimensional modules, the finite direct sums of colinear modules $\Ia{\alpha}$ attached
to the arrows $\alpha$ of the quiver (Theorem~\ref{thm:Fcp-characterization}), whose images
under $\FF$ are the cyclic modules $\La\alpha$. Consequently, the cotails category is
semisimple, and every object of $\Dsg(\La\gmod)$
is a direct sum of shifts of objects of the heart of the singular $t$-structure
(Theorem~\ref{thm:t-monomial}). We also show that $\mathcal L(\La)=\mathcal K(\La)$
(Corollary~\ref{cor:mon-linear-pres}) and that $\findim\La\le\grfindim\La+2$
(Corollary~\ref{cor:findim-grfindim}).

\medskip

\noindent\textbf{Organisation.} Section~\ref{sec:prelim} fixes notation and conventions, following
\cite{Bo1}, and collects preliminaries on minimal complexes, linear parts and co-coherence.
In Section~\ref{sec:linear-part} we describe the minimal projective resolutions of objects of
$\Db(\La\gmod)$ in terms of the Koszul dual and derive first consequences.
Section~\ref{sec:applications} contains the applications: absolute Koszulity and
co-coherence, the graded finitistic dimension, the refined Koszul dualities and the induced
$t$-structures, and modules with linear presentations. Finally, Section~\ref{sec:monomial}
is devoted to finite-dimensional quadratic monomial algebras.

\medskip

\noindent\textbf{Acknowledgments.}
Most of the ideas and results of this paper originated during the author's visit to the University of South Florida in the spring of 2020. The author thanks M.~Elhamdadi for his warm hospitality. The study of quadratic monomial algebras began in 2018, at the end of the author's Ph.D. studies. These algebras provided the first example motivating the project \emph{A non-graded Koszul duality and its applications} \cite{Bo1}, whose aim was to generalize results of Bautista and Liu on algebras with radical square zero \cite{BL}.

\section{Preliminaries}\label{sec:prelim}

In this section we fix the notation and recall from \cite{Bo1} the facts on Koszul
functors that will be used throughout the paper. We also collect some elementary results
on minimal complexes of graded projective modules, their linear parts and the linearity
defect, and on finitely copresented modules over the Koszul dual. 

\subsection{Koszul algebras and their modules}\label{ss:conventions}
We follow the conventions of \cite{Bo1}. Throughout, $k$ is a field, $Q$ a finite quiver
and $\La=kQ/I$ a finite-dimensional Koszul algebra, graded by path length. Since $\La$
is quadratic, $\La_0=\bigoplus_{x\in Q_0}ke_x$ and $\La_1=kQ_1$. We put
\( 
\ell:=\max\{j\ge0\mid \La_j\neq0\}.
\)
The Koszul dual is $\Ld=kQ^{\mathrm{op}}/\langle I_2^{\perp}\rangle$; it is a locally
finite-dimensional Koszul algebra with $\Ld_0=\La_0$, and $(\Ld)^!\cong\La$
\cite[\S2.2]{Bo1}.

We denote by $\La\fmod$ the category of finite-dimensional $\La$-modules, that is,
covariant $k$-linear functors from $\La$, regarded as a $k$-category with object set
$Q_0$, to the category of finite-dimensional $k$-vector spaces.
Graded modules are $k$-linear functors to graded vector spaces; an arrow acts in degree
$+1$. For a graded module $M$ we write $M_q(x)=e_xM_q$. We use $\La\gmod$ (finite-dimensional
graded modules), $\Ld\GMod$ (locally finite-dimensional graded modules) and the full
subcategory $\Ld\GMod^-$ of right bounded modules. The grading shift is
\( 
(M\sh{i})_n=M_{n+i},
\)
and $\DD=\Hom_k(-,k)$ with $(\DD V)_i=\Hom_k(V_{-i},k)$.

For $x\in Q_0$ we denote by $S_x$ and $S^!_x$ the simple $\La$-module and the simple
$\Ld$-module at $x$, both concentrated in degree $0$. The associated indecomposable
projective and injective modules are
\(P_x=\La(x,-)\), \(I_x=\DD\big(\La^{\mathrm{op}}(x,-)\big)\), \(P^{!}_x=\Ld(x,-)\), and \(I^{!}_x=\DD\big(\Ld{}^{\mathrm{op}}(x,-)\big)\).
Thus, for $r\in\mathbb Z$, the modules $P_x\sh r$ and $P^!_x\sh r$ are generated in
degree $-r$, with tops $S_x\sh r$ and $S^!_x\sh r$ respectively, while $I_x\sh r$ and
$I^!_x\sh r$ are cogenerated in degree $-r$, with socles $S_x\sh r$ and $S^!_x\sh r$.
The modules $P_x\sh r$ are projective and the modules $I_x\sh r$ are injective in
$\La\gmod$; likewise, the modules $P^!_x\sh r$ are projective and the modules
$I^!_x\sh r$ are injective in $\Ld\GMod$ \cite[\S3.2]{Bo1}. We write $\La\proj$ and
$\La\inj$ for the categories of finite direct sums of the $P_x\sh r$ and of the
$I_x\sh r$ respectively, and $\Ld\Proj$ and $\Ld\Inj$ for the categories of finite
direct sums of the $P^!_x\sh r$ and of the $I^!_x\sh r$.

For a cochain complex $X$ we use $(X[m])^n=X^{n+m}$ and $d_{X[m]}=(-1)^md_X$.

We begin with an elementary observation on morphisms between the indecomposable objects
of $\La\proj$.

\begin{lemma}\label{lem:hom-proj}
For $x,y\in Q_0$ and $a,b\in\mathbb Z$ one has
$\Hom_{\La\gmod}(P_x\sh a,P_y\sh b)\cong e_x\La_{b-a}e_y$. In particular this space is zero
unless $0\le b-a\le\ell$, and every map $P_x\sh a\to P_y\sh a$ is a scalar multiple of the
identity when $x=y$ and zero otherwise.
\end{lemma}

\subsection{The Koszul functor and the derived quadratic functor}
\label{ss:koszul-functor}
In this subsection we recall from \cite[\S3.1]{Bo1} the Koszul functor $K$ (see also
\cite{BHL,BHLL}), which sends a
graded $\Ld$-module to a linear complex in $\La\proj$, and the derived quadratic functor
$\FF$, obtained by applying $K$ termwise to complexes, together with the properties of
these functors that will be used in the sequel. For $M\in\Ld\GMod$, the complex $K(M)$ has
terms and differential
\begin{equation*}
K(M)^n=\bigoplus_{x\in Q_0}P_x\sh{n}\otimes_k M_n(x),\qquad
d^n_{yx}=\sum_{\alpha\colon y\to x}P_\alpha\otimes M(\alpha^{\mathrm{op}}),
\end{equation*}
where $P_\alpha\colon P_x\sh n\to P_y\sh{n+1}$ is right multiplication by $\alpha$ and
$M(\alpha^{\mathrm{op}})\colon M_n(x)\to M_{n+1}(y)$ is the action on $M$ of the arrow
$\alpha^{\mathrm{op}}\colon x\to y$ of $Q^{\mathrm{op}}$ opposite to $\alpha$. A complex
$L$ in $\C(\La\proj)$ is \emph{linear} if $L^n\in\add\{P_x\sh n\mid x\in Q_0\}$ for every
$n$, i.e.\ $L^n$ is generated in degree $-n$. By the construction in \cite[\S3.1]{Bo1}, see also
\cite[Thm.~2.4]{MS}, \cite[Thm.~12]{MOS}, $K$ is an exact, fully faithful functor whose
essential image is the category $\LC(\La\proj)$ of linear complexes, and there are no
nonzero null-homotopic maps between linear complexes \cite[Lemma~3.9]{Bo1}. A direct
computation gives a natural isomorphism
\begin{equation}\label{eq:K-shift}
K\big(M\sh i\big)\cong K(M)\sh{-i}[i]\qquad(i\in\mathbb Z).
\end{equation}

For a bounded complex $X$ of graded $\Ld$-modules, $\Kd(X)$ is the double complex with
\begin{equation*}
\Kd(X)^{p,q}=\bigoplus_{x\in Q_0}P_x\sh q\otimes_k X^p_q(x),\qquad
d_1=\id\otimes d_X\ (\text{horizontal}),\qquad
d_2=\sum_\alpha P_\alpha\otimes X^p(\alpha^{\mathrm{op}})\ (\text{vertical}),
\end{equation*}
and $\Fc(X)=\Tot\Kd(X)$ has $\Fc(X)^n=\bigoplus_{p+q=n}\Kd(X)^{p,q}$ and differential
$d_1+(-1)^pd_2$. In \cite{Bo1} this is considered for complexes of modules in
$\Ld\GMod^{-,b}$ (see below); we use the same formula for every bounded complex $X$ in
$\Ld\GMod^-$. Since $X$ is bounded and each $X^p$ is right bounded and locally
finite-dimensional, $\Fc(X)$ is a right bounded complex of objects of $\La\proj$. If
$X$ is a complex with zero differential, then
\begin{equation}\label{eq:F-of-H}
\Fc(X)=\bigoplus_{p}K(X^p)[-p].
\end{equation}

A module $M\in\Ld\GMod^-$ belongs to $\Ld\GMod^{-,b}$ if $K(M)$ has bounded cohomology.
Since $K$ is exact, $\Ld\GMod^{-,b}$ has the \emph{two-out-of-three property}: if
$0\to A\to B\to C\to 0$ is exact in $\Ld\GMod^-$ and two of its terms lie in
$\Ld\GMod^{-,b}$, so does the third.

Following \cite[\S3.2]{Bo1}, a module $M\in\Ld\GMod^-$ is \emph{coperfect} if it admits a
finite injective coresolution with all terms in $\Ld\Inj$. We write $\Ld\Cop$ for the full
subcategory of $\Ld\GMod^-$ consisting of the coperfect modules.

We shall use the following results of \cite{Bo1}.

\begin{theorem}[{\cite[Proposition~3.5 and Theorem~3.10]{Bo1}}]\label{thm:Bo1}
Let $\La$ be a finite-dimensional Koszul algebra.
\begin{enumerate}[label=\textup{(\alph*)}]
\item $\Ld\Cop=\Ld\GMod^{-,b}$.
\item $\Fc$ induces a triangulated equivalence
$\FF=\TT\circ\Fc\colon\Db(\Ld\Cop)\xrightarrow{\sim}\Db(\La\gmod)$, where
$\TT\colon\Kb^{-,b}(\La\proj)\xrightarrow{\sim}\Db(\La\gmod)$ is the canonical equivalence.
\item $\FF(I^!_x\sh n)\cong S_x\sh{-n}[n]$ and $\FF(S^!_x\sh n)\cong P_x\sh{-n}[n]$; indeed
$K(I^!_x)$ is the minimal graded projective resolution of $S_x$.
\item Since $\Ld$ has finite global dimension and its injectives are strongly injective,
every object of $\Db(\Ld\Cop)$ is isomorphic to a bounded complex in $\Ld\Inj$, and
$\Db(\Ld\Cop)$ is a full subcategory of $\Db(\Ld\GMod)$.
\end{enumerate}
\end{theorem}

We fix a quasi-inverse $\GG$ of $\FF$; by Theorem~\ref{thm:Bo1}(d) it may be chosen so that
$\GG(Y)$ is a bounded complex in $\Ld\Inj$ for every $Y\in\Db(\La\gmod)$. Since
$\Db(\Ld\Cop)$ is a full subcategory of $\Db(\Ld\GMod)$ by Theorem~\ref{thm:Bo1}(d), the
cohomology modules $H^p(\GG(Y))\in\Ld\GMod^-$ are well defined. We shall not need an
explicit description of $\GG$: all our computations are carried out with $\FF$, applied to
a bounded complex $X$ with $\FF(X)\cong Y$.

\subsection{Finitely copresented modules and co-coherence}\label{ss:cocoherence}
Following \cite[\S3.2]{Bo1}, a module $M\in\Ld\GMod^-$ is \emph{finitely copresented}
if there is an exact sequence $0\to M\to J^0\to J^1$ with $J^0,J^1\in\Ld\Inj$. We write
$\Ld\Fcp$ for the corresponding full subcategory of $\Ld\GMod^-$. Dually, we denote by $\Ld\Fp$ the full subcategory of $\Ld\GMod$ consisting of the
\emph{finitely presented} graded modules. Every nonzero right bounded module has a nonzero essential socle
\cite[\S3.2]{Bo1}; hence a right bounded module with finite-dimensional socle has an
injective envelope in $\Ld\Inj$.

\begin{definition}\label{def:cocoherent}
We say that $\Ld$ is \emph{graded left co-coherent} if $\Ld\Fcp$ is an abelian
subcategory of $\Ld\GMod$.
\end{definition}

The following vanishing result will be used to bound the length of injective
coresolutions of finitely copresented modules.

\begin{lemma}\label{lem:ext-vanishing}
For every $x\in Q_0$, $m\in\mathbb Z$, $M\in\Ld\GMod$ and $i>\ell$ one has
$\Ext^i_{\Ld\GMod}(S^!_x\sh m,M)=0$.
\end{lemma}

\begin{proof}
Since $\Ld$ is Koszul with Koszul dual $\La$, the minimal graded projective resolution of
$S^!_x$ is linear and its $i$-th term is a direct sum of copies of $P^!_y\sh{-i}$,
$y\in Q_0$, with multiplicities $\dim\Ext^i(S^!_x,S^!_y\sh{-i})$ given by the dimensions of
the components of $\La_i$ \cite[\S2.2 and Theorem~2.4]{Bo1}. Hence this resolution has length
$\ell$. Its terms are projective in $\Ld\GMod$, so it computes
$\Ext^\ast_{\Ld\GMod}(S^!_x\sh m,-)$.
\end{proof}

The next lemma relates the co-coherence of $\Ld$ to the coperfect category.

\begin{lemma}\label{lem:cocoherent}
\begin{enumerate}[label=\textup{(\alph*)}]
\item $\Ld\Cop\subseteq\Ld\Fcp$, and $\Ld\Fcp$ is closed under kernels in $\Ld\GMod$.
\item The following conditions are equivalent:
\begin{enumerate}[label=\textup{(\roman*)}]
\item $\Ld$ is graded left co-coherent;
\item $\Ld\Cop$ is abelian;
\item $\Ld\Fcp\subseteq\Ld\Cop$.
\end{enumerate}
When they hold, $\Ld\Fcp=\Ld\Cop$.
\end{enumerate}
\end{lemma}

\begin{proof}
(a) The first assertion is clear by truncating a coresolution. Let $f\colon A\to B$ be a
morphism in $\Ld\Fcp$, with copresentations $0\to A\to I^0\xrightarrow{c}I^1$ and
$0\to B\xrightarrow{j}J^0\to J^1$. Since $J^0$ is injective, $jf$ extends to a map
$u\colon I^0\to J^0$. If $z\in I^0$ satisfies $c(z)=0$ and $u(z)=0$, then $z\in A$ and
$jf(z)=0$, so $f(z)=0$; conversely $\Ker f\subseteq\Ker c\cap\Ker u$. Hence
$0\to\Ker f\to I^0\xrightarrow{(c,u)}I^1\oplus J^0$ is a copresentation, and
$\Ker f\in\Ld\Fcp$.

(b) (iii)$\Rightarrow$(i),(ii): by (a), $\Ld\Fcp=\Ld\Cop$. Let $f\colon A\to B$ be a
morphism in it. By (a), $\Ker f\in\Ld\Cop$. All modules involved are right bounded, so
the two-out-of-three property of $\Ld\Cop=\Ld\GMod^{-,b}$ (Theorem~\ref{thm:Bo1}(a)),
applied to $0\to\Ker f\to A\to\Img f\to0$ and then to $0\to\Img f\to B\to\Coker f\to0$,
gives $\Img f,\Coker f\in\Ld\Cop$.

(ii)$\Rightarrow$(iii): if $0\to M\to J^0\to J^1$ is a copresentation, then $M$ is the
kernel of a morphism in $\Ld\Cop$, hence $M\in\Ld\Cop$.

(i)$\Rightarrow$(iii): let $M\in\Ld\Fcp$. Since $M\subseteq J^0$ with $J^0\in\Ld\Inj$,
the socle of $M$ is finite-dimensional, so $M$ has an injective envelope
$M\hookrightarrow E^0$ with $E^0\in\Ld\Inj$. Its cokernel lies in $\Ld\Fcp$ because
$\Ld\Fcp$ is closed under cokernels. Iterating, we obtain a minimal injective coresolution
$0\to M\to E^0\to E^1\to\cdots$ with all $E^i\in\Ld\Inj$, in which each cosyzygy is
essential in $E^i$. Consequently, every morphism $S^!_x\sh m\to E^i$ lands in the socle of
the $i$-th cosyzygy and composes to zero with $E^i\to E^{i+1}$, so
\begin{equation*}
\Ext^i_{\Ld\GMod}\big(S^!_x\sh m,M\big)\cong\Hom\big(S^!_x\sh m,E^i\big).
\end{equation*}
By Lemma~\ref{lem:ext-vanishing} this vanishes for $i>\ell$, so $\soc E^i=0$ and hence
$E^i=0$ for $i>\ell$. Thus $M\in\Ld\Cop$.

The final assertion follows from (a) and (b).
\end{proof}

We now introduce the linear part of a minimal complex, which is the main tool of this paper.

\subsection{Minimal complexes and their linear parts}\label{ss:linear-part}
Every $P\in\La\proj$ admits a \emph{weight decomposition}
$P=\bigoplus_{a\in\mathbb Z}P_{(a)}$, where
$P_{(a)}\in\add\{P_x\sh a\mid x\in Q_0\}$. Such a decomposition is not
canonical: the submodules $P_{(a)}$ depend on the chosen splitting, but the
multiplicity of each $P_x\sh a$ is uniquely determined by the Krull--Schmidt
theorem. Indeed, these multiplicities can be read off from the graded top
$P/\rad P$. Let $P,P'\in\La\proj$ be equipped with weight decompositions and
$f\colon P\to P'$ a morphism. For $j\in\mathbb Z$, let $f_{[j]}$ be the sum
of the components $P_{(a)}\to P'_{(a+j)}$ of $f$. By
Lemma~\ref{lem:hom-proj},
\begin{equation}\label{eq:weights}
f=\sum_{j=0}^{\ell}f_{[j]},\qquad
(gf)_{[j]}=\sum_{i+i'=j}g_{[i]}f_{[i']}.
\end{equation}
The components $f_{[j]}$ with $j\ge1$ take values in
$\La_{\ge1}P'=\rad P'$, while $f_{[0]}$ induces the same map as $f$ on
tops and is determined by it. Hence
\begin{equation}\label{eq:radical}
f(P)\subseteq\rad P'\iff f_{[0]}=0.
\end{equation}

\begin{definition}\label{def:minimal}
A complex $C\in\C^-(\La\proj)$ is \emph{minimal} if
$d(C^n)\subseteq\rad C^{n+1}$ for all $n$. Choose weight decompositions of
the terms of $C$. By \eqref{eq:radical}, we have $d_{[0]}=0$, and the
weight-two part of $d^2=0$ reads $d_{[1]}^2=0$ by
\eqref{eq:weights}. The \emph{linear part} of $C$ is the complex
$\Lin(C)=\big(C^n,d_{[1]}\big)_{n\in\mathbb Z}$.
\end{definition}

We record the basic properties of linear parts for later use.

\begin{lemma}\label{lem:linear-part}
Let $C$ and $C'$ be minimal complexes in $\C^-(\La\proj)$ with weight
decompositions.
\begin{enumerate}[label=\textup{(\alph*)}]
\item If $\varphi\colon C\to C'$ is a homotopy equivalence, then
$\varphi$ is an isomorphism of complexes, and
$\varphi_{[0]}\colon\Lin(C)\to\Lin(C')$ is an isomorphism of complexes.
In particular, $\Lin(C)$ does not depend, up to isomorphism, on the
chosen weight decompositions.
\item $\Lin(C)=\bigoplus_{s\in\mathbb Z}\Lin(C)_s$, where
$\Lin(C)_s^n=C^n_{(n+s)}$, and $\Lin(C)_s\sh{-s}$ is a linear complex.
We call $\Lin(C)_s$ the \emph{linear strands} of $C$.
\item Every linear complex $L$ is minimal and $\Lin(L)=L$.
\end{enumerate}
\end{lemma}

\begin{proof}
(a) Choose a homotopy inverse $\psi\colon C'\to C$ of $\varphi$.
Since $C$ and $C'$ are minimal, taking weight-zero components of the
homotopy identities gives
$\psi_{[0]}\varphi_{[0]}=1_C$ and
$\varphi_{[0]}\psi_{[0]}=1_{C'}$. For each $n$, put
$u=\psi^n_{[0]}\sum_{j\ge1}\varphi^n_{[j]}$, so that
$\varphi^n=\varphi^n_{[0]}(1+u)$. Every nonzero component of $u$
strictly increases weight. If the weights occurring in $C^n$ lie in
$[a,b]$, then $u^{b-a+1}=0$. Hence $1+u$ has inverse
$\sum_{r=0}^{b-a}(-u)^r$, and $\varphi^n$ is invertible.
Thus $\varphi$ is an isomorphism of complexes. Taking weight-one
components of $\varphi d_C=d_{C'}\varphi$ gives
$\varphi_{[0]}(d_C)_{[1]}=(d_{C'})_{[1]}\varphi_{[0]}$.
It follows that $\varphi_{[0]}$ is an isomorphism
$\Lin(C)\to\Lin(C')$, with inverse $\psi_{[0]}$.
Applying this to $\id_C$ with two choices of weight decomposition
proves the last assertion.

(b) The map $d_{[1]}$ sends $C^n_{(a)}$ to
$C^{n+1}_{(a+1)}$ and therefore preserves $a-n$.
Since $P_x\sh{n+s}\sh{-s}=P_x\sh n$,
$\Lin(C)_s\sh{-s}$ is linear.

(c) In degree $n$, every term of a linear complex has weight $n$.
Thus its differential has weight one, so it is minimal and equals
its own linear part.
\end{proof}

For $N\in\La\gmod$, let $P_N$ denote its minimal graded projective
resolution, concentrated in non-positive degrees.

\begin{definition}\label{def:ld}
For $N\in\La\gmod$, its \emph{linearity defect} is
$\ld(N)=\sup\{i\ge0\mid H^{-i}(\Lin P_N)\ne0\}$, with the convention
$\sup\varnothing=-\infty$. It is well defined by
Lemma~\ref{lem:linear-part}(a). The \emph{global linearity defect} is
$\glld\La=\sup\{\ld(N)\mid N\in\La\gmod\}$. A nonzero module $N$ is
\emph{weakly Koszul} if $\ld(N)=0$, equivalently, if $\Lin(P_N)$ is
exact in negative degrees. We also regard the zero module as weakly
Koszul. The algebra $\La$ is \emph{absolutely Koszul} if
$\ld(N)<\infty$ for every $N\in\La\gmod$.

More generally, a \emph{minimal projective resolution} of an object
$Y\in\Db(\La\gmod)$ is a minimal complex
$C_Y\in\C^-(\La\proj)$ together with an isomorphism
$\TT(C_Y)\cong Y$ in $\Db(\La\gmod)$.
Minimal projective resolutions exist (Corollary~\ref{cor:linear-part})
and are unique up to isomorphism of complexes
(Lemma~\ref{lem:linear-part}(a)); for a module $N$, $C_N=P_N$.
We say that $Y$ has \emph{finite linearity defect} if
$\Lin(C_Y)$ has bounded cohomology. Since $\Lin(P_N)$ is right bounded,
a module $N$ has finite linearity defect in this sense if and only if
$\ld(N)<\infty$.
\end{definition}

As in the commutative case \cite{IR}, the linearity defect can be read
off from sufficiently high syzygies.

\begin{lemma}\label{lem:syzygy}
Let $N\in\La\gmod$ and $d\ge0$. Then
$H^n(\Lin P_{\Omega^dN})\cong H^{n-d}(\Lin P_N)$ for all $n\le-1$.
In particular, $\ld(N)\le d$ if and only if $\Omega^dN$ is weakly
Koszul or zero, and $\La$ is absolutely Koszul if and only if every
module has a weakly Koszul syzygy.
\end{lemma}

\begin{proof}
The minimal projective resolution of $\Omega^dN$ is
$P_{\Omega^dN}\cong(\sigma^{\le-d}P_N)[-d]$, where
$\sigma^{\le-d}$ denotes brutal truncation. The same relation holds
for the linear parts, since truncation and cohomological shifting
preserve the weight-one components of the differentials. Hence, for
$n\le-1$,
$H^n(\Lin P_{\Omega^dN})\cong
H^{n-d}(\sigma^{\le-d}\Lin P_N)
\cong H^{n-d}(\Lin P_N)$,
as brutal truncation does not change cohomology in degrees below
$-d$. The stated equivalence follows by vanishing in all negative
degrees. The last assertion follows by allowing $d$ to vary with $N$.
\end{proof}

\subsection{Koszul and coKoszul modules}\label{ss:koszul-modules}
A module $M\in\La\gmod$ is \emph{Koszul} if $P_M$ is linear; equivalently, $M$ is
generated in degree $0$ and has a linear projective resolution. We write $\mathcal K(\La)$
for the category of Koszul modules. Thus $\La$ is Koszul if and only if every $S_x$ is
Koszul.

\begin{definition}\label{def:cokoszul}
A module $C\in\Ld\GMod^-$ is \emph{coKoszul} if $H^n(K(C))=0$ for all $n\neq0$.
\end{definition}

The following lemma shows that coKoszul modules correspond to Koszul modules under the
Koszul functor.

\begin{lemma}\label{lem:cokoszul}
\begin{enumerate}[label=\textup{(\alph*)}]
\item Let $C\ne0$ be coKoszul. Then $C_n=0$ for $n>0$, $\soc C=C_0\ne0$, $C\in\Ld\Cop$, and
$K(C)$ is the minimal graded projective resolution of the Koszul module $H^0(K(C))$.
\item Conversely, if $M\ne0$ is Koszul, then $P_M\cong K(C_M)$ for a coKoszul module
$C_M$, unique up to isomorphism.
\item If $M$ and $C=C_M$ correspond as in \textup{(a)} and \textup{(b)}, then $\pd M<\infty$ if and only if
$C$ is finite-dimensional, and then
\begin{equation*}
\pd M=\lgr(C)-1,\qquad \lgr(C):=\max\{n\mid C_n\ne0\}-\min\{n\mid C_n\ne0\}+1.
\end{equation*}
\end{enumerate}
\end{lemma}

\begin{proof}
(a) The complex $K(C)$ is linear, hence minimal, and right bounded since $C$ is. Let $t$
be the largest integer with $K(C)^t\ne0$. If $t>0$, then $H^t(K(C))=0$ means that
$d^{t-1}$ is surjective; but its image lies in $\rad K(C)^t$, which contradicts graded
Nakayama. Hence $K(C)^n=0$, that is $C_n=0$, for $n>0$. If $0\ne u\in\soc(C)_n(x)$ with
$n<0$, then $e_x\otimes u\in K(C)^n$ is a cycle \cite[proof of Proposition~3.5]{Bo1} which does
not lie in $\rad K(C)^n$, hence it is not a boundary, and $H^n(K(C))\ne0$, a contradiction.
Thus $\soc C\subseteq C_0$; conversely $C_0\subseteq\soc C$ since $C_1=0$, and $\soc C\neq0$.
The module $H^0(K(C))$ is a quotient of $K(C)^0=\bigoplus_xP_x\sh0\otimes C_0(x)$, hence it
is generated in degree $0$, and $K(C)$ is a linear resolution of it. Finally $K(C)$ has
bounded cohomology, so $C\in\Ld\GMod^{-,b}=\Ld\Cop$.

(b) $P_M$ is a linear complex with finitely generated terms, so $P_M\cong K(C)$ for some
$C\in\Ld\GMod$, unique up to isomorphism because $K$ is fully faithful. Since
$P_M^n=0$ for $n>0$, $C$ is right bounded, and it is coKoszul.

(c) $K(C)^n\ne0$ if and only if $C_n\ne0$. Since $C_0\ne0$ and $C_n=0$ for $n>0$, the
length of $P_M=K(C)$ equals $-\min\{n\mid C_n\neq 0\}=\lgr(C)-1$.
\end{proof}

\section{The linear part via graded derived Koszul duality}\label{sec:linear-part}

The aim of this section is to describe the minimal projective resolution of an object $Y$
of $\Db(\La\gmod)$ in terms of the complex $\GG(Y)$ of coperfect graded $\Ld$-modules.
Starting from a bounded complex $X$ with $\FF(X)\cong Y$, we contract the horizontal
differential of the double complex $\Kd(X)$ and regard the vertical differential as a
perturbation. The resulting complex is minimal, and its linear part only depends on the
cohomology of $X$.

\subsection{The perturbation lemma}
We use the basic perturbation lemma of homological algebra; see \cite{Br} and
\cite[\S2]{Cr}. Let $(M,b)$ and $(L,b_L)$ be complexes of graded $\La$-modules, and let
$\iota\colon L\to M$, $\pi\colon M\to L$ be chain maps and $h\colon M\to M$ a map of
degree $-1$ such that
\begin{equation}\label{eq:sdr}
\pi\iota=1_L,\qquad \iota\pi=1_M+bh+hb,\qquad h\iota=0,\qquad \pi h=0,\qquad h^2=0.
\end{equation}

\begin{lemma}[Perturbation lemma]\label{lem:perturbation}
Let $\delta\colon M\to M$ be a map of degree $+1$ with $(b+\delta)^2=0$, and assume that
$\delta h$ is nilpotent. Put $A=\sum_{r\ge0}(\delta h)^r\delta$ and
\begin{equation*}
b_L'=b_L+\pi A\iota,\quad \iota'=\iota+hA\iota,\quad \pi'=\pi+\pi Ah,\quad h'=h+hAh.
\end{equation*}
Then $(b_L')^2=0$, $\iota'$ and $\pi'$ are chain maps between $(L,b_L')$ and $(M,b+\delta)$,
and the data $(\iota',\pi',h')$ satisfy \eqref{eq:sdr} with respect to $b+\delta$ and
$b_L'$. In particular, $(L,b'_L)$ and $(M,b+\delta)$ are homotopy equivalent.
\end{lemma}

\subsection{The structure theorem}
Let $X$ be a bounded complex in $\Ld\GMod^-$. Its cohomology modules $H^p(X)$ lie in
$\Ld\GMod^-$, and we let $H^\bullet(X)$ be the complex with zero differential and
$H^\bullet(X)^p=H^p(X)$. By \eqref{eq:F-of-H},
\begin{equation*}
\Fc\big(H^\bullet(X)\big)=\bigoplus_{p\in\mathbb Z}K\big(H^p(X)\big)[-p],\qquad
\Fc\big(H^\bullet(X)\big)^n=\bigoplus_{p+q=n}\ \bigoplus_{x\in Q_0}P_x\sh q\otimes H^p(X)_q(x).
\end{equation*}
We endow these terms, and those of $\Fc(X)$, with the evident weight decompositions:
the summand $P_x\sh q\otimes V$ has weight $q$.

The following is the main result of this section.

\begin{theorem}\label{thm:structure}
Let $X$ be a bounded complex in $\Ld\GMod^-$. There is a differential $D$ on the graded
$\La$-module $C(X):=\Fc(H^\bullet(X))$ with the following properties.
\begin{enumerate}[label=\textup{(\alph*)}]
\item The complex $(C(X),D)$ is homotopy equivalent to $\Fc(X)$.
\item $(C(X),D)$ is minimal.
\item $D_{[1]}$ is the differential of $\Fc(H^\bullet(X))$; that is, as complexes,
\begin{equation*}
\Lin\big(C(X),D\big)=\bigoplus_{p\in\mathbb Z}K\big(H^p(X)\big)[-p].
\end{equation*}
Moreover, the linear strand of weight $s$ is $K(H^{-s}(X))[s]$.
\item More precisely, $D=\sum_{r\ge0}D_{[r+1]}$ with $D_{[r+1]}=\pi(\delta h)^r\delta\iota$,
where $\delta=(-1)^pd_2$ and $\iota,\pi,h$ are the maps of weight zero constructed in the
proof.
\end{enumerate}
\end{theorem}

\begin{proof}
Write $M=\Fc(X)$, with differential $b+\delta$, where $b=d_1=\id\otimes d_X$ and
$\delta=(-1)^pd_2$ on $\Kd(X)^{p,q}$. Both preserve the decomposition
\begin{equation*}
M=\bigoplus_{q\in\mathbb Z}\ \bigoplus_{x\in Q_0}P_x\sh q\otimes V_{x,q},\qquad
V_{x,q}:=\big(X^p_q(x)\big)_{p\in\mathbb Z},
\end{equation*}
in the following sense: $b=\bigoplus\id_{P_x\sh q}\otimes d_{x,q}$, where $d_{x,q}$ is the
differential of the bounded complex of finite-dimensional vector spaces $V_{x,q}$, while
$\delta$ maps $P_x\sh q\otimes V^p_{x,q}$ to
$\bigoplus_yP_y\sh{q+1}\otimes V^p_{y,q+1}$. Note that $H^p(V_{x,q})=H^p(X)_q(x)$, because
cohomology of graded modules is computed degreewise and vertexwise.

We first contract each complex $V_{x,q}$ onto its cohomology. For each $(x,q)$ and $p$ we
choose decompositions
\begin{equation*}
V^p_{x,q}=H'^{\,p}\oplus B^p\oplus E^p,\qquad Z^p=H'^{\,p}\oplus B^p,
\end{equation*}
where $Z^p=\Ker d^p_{x,q}$, $B^p=\Img d^{p-1}_{x,q}$ and $E^p$ is a complement of $Z^p$, so
that $d_{x,q}$ maps $E^p$ isomorphically onto $B^{p+1}$. Let $\iota_{x,q}\colon
H^p(V_{x,q})\to V^p_{x,q}$ be the inverse of the isomorphism $H'^{\,p}\to H^p(V_{x,q})$
followed by the inclusion, let $\pi_{x,q}\colon V^p_{x,q}\to H^p(V_{x,q})$ vanish on
$B^p\oplus E^p$ and restrict to the canonical map on $H'^{\,p}$, and let
$h_{x,q}\colon V^p_{x,q}\to V^{p-1}_{x,q}$ be $-(d_{x,q}|_{E^{p-1}})^{-1}$ on $B^p$ and zero
on $H'^{\,p}\oplus E^p$. A direct check shows that these maps satisfy the relations
\eqref{eq:sdr} for $(V_{x,q},d_{x,q})$ and $(H^\bullet(V_{x,q}),0)$. Moreover, the
restriction of $\pi_{x,q}$ to $Z^p$ is the canonical projection $Z^p\to H^p(V_{x,q})$, and
$\iota_{x,q}$ takes values in $Z^p$.

Tensoring these data with the identity maps of the modules $P_x\sh q$, we obtain the maps
\begin{equation*}
\iota=\bigoplus_{x,q}\id_{P_x\sh q}\otimes\iota_{x,q},\quad
\pi=\bigoplus_{x,q}\id_{P_x\sh q}\otimes\pi_{x,q},\quad
h=\bigoplus_{x,q}\id_{P_x\sh q}\otimes h_{x,q}.
\end{equation*}
They are $\La$-linear and of weight zero, and they satisfy \eqref{eq:sdr} between $(M,b)$ and
$(C(X),0)$, where $C(X)=\bigoplus_{x,q}P_x\sh q\otimes H^\bullet(V_{x,q})=\Fc(H^\bullet(X))$
as graded modules, with the cohomological degree $p+q$.

We now regard $\delta$ as a perturbation of $b$. The map $\delta h$ sends $\Kd(X)^{p,q}$ to
$\Kd(X)^{p-1,q+1}$, so if $X^p=0$ outside an interval of length $t$, then
$(\delta h)^{t+1}=0$. Lemma~\ref{lem:perturbation}, applied with $b_L=0$, therefore shows
that $D:=\pi A\iota=\sum_{r\ge0}\pi(\delta h)^r\delta\iota$ is a differential on $C(X)$,
and that $(C(X),D)$ is homotopy equivalent to $(M,b+\delta)=\Fc(X)$. Since all the maps
involved are $\La$-linear and all the sums are finite, this proves (a).

Since $\iota$, $\pi$ and $h$ have weight zero and $\delta$ has weight one, the map
$\pi(\delta h)^r\delta\iota$ is homogeneous of weight $r+1$. Hence $D_{[0]}=0$ and
$D_{[r+1]}=\pi(\delta h)^r\delta\iota$, and $(C(X),D)$ is minimal by \eqref{eq:radical}.
This proves (b) and (d).

It remains to compute the weight-one part of $D$. On $P_x\sh q\otimes H^p(X)_q(x)$ we have
\begin{equation*}
D_{[1]}=\pi\delta\iota=(-1)^p\sum_{\alpha\colon y\to x}P_\alpha\otimes
\pi_{y,q+1}\circ X^p(\alpha^{\mathrm{op}})\circ\iota_{x,q}.
\end{equation*}
Since $d_X$ is $\Ld$-linear, $X^p(\alpha^{\mathrm{op}})$ maps cycles to cycles and induces
$H^p(X)(\alpha^{\mathrm{op}})$ on cohomology. As $\iota_{x,q}$ takes values in cycles and
the restriction of $\pi_{y,q+1}$ to cycles is the canonical projection, we obtain
\begin{equation*}
\pi_{y,q+1}\circ X^p(\alpha^{\mathrm{op}})\circ\iota_{x,q}=H^p(X)(\alpha^{\mathrm{op}}).
\end{equation*}
Consequently $D_{[1]}=(-1)^p\sum_\alpha P_\alpha\otimes H^p(X)(\alpha^{\mathrm{op}})$, which
is the differential of $\bigoplus_pK(H^p(X))[-p]$. Finally, the summand $K(H^p(X))[-p]$ has,
in degree $n$, the terms $P_x\sh{n-p}\otimes H^p(X)_{n-p}(x)$ of weight $n-p$, so it is the
strand $s=-p$ of Lemma~\ref{lem:linear-part}(b). This proves (c).
\end{proof}

\subsection{Consequences}

We now apply Theorem~\ref{thm:structure} to the objects of $\Db(\La\gmod)$.

\begin{corollary}\label{cor:linear-part}
Let $Y\in\Db(\La\gmod)$ and $X\in\C^b(\Ld\Cop)$ with $\FF(X)\cong Y$. Then $(C(X),D)$ is a
minimal projective resolution of $Y$, and
\begin{equation*}
\Lin(C_Y)\;\cong\;\bigoplus_{p\in\mathbb Z}K\big(H^p(X)\big)[-p]
\end{equation*}
as complexes. In particular, every object of $\Db(\La\gmod)$ has a minimal projective resolution, and for
$N\in\La\gmod$
\begin{equation*}
P_N\cong\big(C(X),D\big),\qquad
P_N^n\cong\bigoplus_{p+q=n}\ \bigoplus_{x\in Q_0}P_x\sh q\otimes H^p(X)_q(x),\qquad
\Lin(P_N)\cong\bigoplus_{p}K\big(H^p(X)\big)[-p].
\end{equation*}
\end{corollary}

\begin{proof}
Such an $X$ exists by Theorem~\ref{thm:Bo1}. Since $X\in\C^b(\Ld\GMod^{-,b})$, the complex
$\Fc(X)$ lies in $\Kb^{-,b}(\La\proj)$ and $\TT(\Fc(X))=\FF(X)\cong Y$. By
Theorem~\ref{thm:structure}(a), $(C(X),D)$ also lies in $\Kb^{-,b}(\La\proj)$ and
$\TT(C(X),D)\cong Y$, and it is minimal by Theorem~\ref{thm:structure}(b). If $C_Y$ is
another minimal projective resolution of $Y$, then $C_Y$ and $(C(X),D)$ are isomorphic in $\Kb^{-,b}(\La\proj)$,
since $\TT$ is an equivalence, hence isomorphic as complexes with isomorphic linear parts
by Lemma~\ref{lem:linear-part}(a). For $Y=N$ a module, $P_N$ is the minimal projective resolution of $N$.
The rest follows from Theorem~\ref{thm:structure}(c).
\end{proof}

The cohomology of $X$ only depends on the isomorphism class of $\GG(Y)$ in
$\Db(\Ld\Cop)$, by Theorem~\ref{thm:Bo1}(d). We shall therefore write $H^p(\GG(Y))$.

\begin{corollary}\label{cor:ext}
Let $N\in\La\gmod$. For all $x\in Q_0$, $m\in\mathbb Z$ and $i\ge0$ there are natural
isomorphisms
\begin{equation*}
\Ext^i_{\La\gmod}\big(N,S_x\sh m\big)\cong\DD\Big(H^{-i-m}(\GG(N))_m(x)\Big).
\end{equation*}
\end{corollary}

\begin{proof}
Since $P_N$ is minimal, $\Hom(P_N,S_x\sh m)$ has zero differential, so
$\Ext^i(N,S_x\sh m)$ is isomorphic to $\Hom(P_N^{-i},S_x\sh m)$. Moreover
$\Hom(P_y\sh q\otimes V,S_x\sh m)\cong\DD V$ if $(y,q)=(x,m)$ and is zero otherwise. By
Corollary~\ref{cor:linear-part}, the summands of $P_N^{-i}$ of the form $P_x\sh m\otimes V$
are those with $q=m$ and $p=-i-m$, with $V=H^{-i-m}(\GG(N))_m(x)$.
\end{proof}

For a complex $X$ in $\Ld\GMod^-$ put
\begin{equation*}
\Sigma(X)=\{p+q\mid H^p(X)_q\ne0\}\subseteq\mathbb Z,
\qquad
\CHL(X)=\sup\Sigma(X)-\inf\Sigma(X)+1,
\end{equation*}
the \emph{graded cohomological length} of $X$.

\begin{corollary}\label{cor:pd}
Let $0\ne N\in\La\gmod$. Then $\sup\Sigma(\GG(N))=0$ and
\begin{equation*}
\pd N=-\inf\Sigma(\GG(N))=\CHL(\GG(N))-1.
\end{equation*}
Moreover, $\pd N<\infty$ if and only if the total cohomology
$\bigoplus_{p}H^p(\GG(N))$ is finite-dimensional.
\end{corollary}

\begin{proof}
By Corollary~\ref{cor:linear-part}, $P_N^n\ne0$ if and only if $n\in\Sigma(\GG(N))$. Since
$P_N^0\ne0$ and $P_N^n=0$ for $n>0$, we get $\sup\Sigma=0$, and
$\pd N=\sup\{i\mid P_N^{-i}\ne0\}=-\inf\Sigma$. Each $H^p(\GG(N))_q(x)$ is
finite-dimensional, and only finitely many $p$ occur; hence $\Sigma(\GG(N))$ is finite if and
only if the total cohomology is finite-dimensional.
\end{proof}

Next we characterise weakly Koszul modules in terms of the Koszul dual.

\begin{corollary}\label{cor:weakly-koszul}
Let $N\in\La\gmod$ and put $C_p:=H^p(\GG(N))\sh{-p}$ for $p\in\mathbb Z$.
\begin{enumerate}[label=\textup{(\alph*)}]
\item $\ld(N)=\sup_p\ \sup\{-j\mid H^j(K(C_p))\ne0\}$.
\item $N$ is weakly Koszul if and only if every $C_p$ is coKoszul.
\item In this case, $\Lin(P_N)\cong\bigoplus_pK(C_p)\sh{-p}$ is the direct sum of the
minimal projective resolutions of the modules $M_p\sh{-p}$, where $M_p:=H^0(K(C_p))$ is
Koszul (or zero). If $N\ne0$, then
$\pd N=\max_{p\colon C_p\ne0}\pd M_p$; when $\pd N<\infty$, this also equals
$\max_{p\colon C_p\ne0}\big(\lgr(C_p)-1\big)$.
\end{enumerate}
\end{corollary}

\begin{proof}
By \eqref{eq:K-shift}, $K(H^p(\GG(N)))[-p]=K(C_p\sh p)[-p]\cong K(C_p)\sh{-p}$, with the same
cohomology up to the shift $\sh{-p}$. Hence, by Corollary~\ref{cor:linear-part},
$H^j(\Lin P_N)\cong\bigoplus_pH^j(K(C_p))\sh{-p}$, which gives (a). Since $\Lin(P_N)$ is
concentrated in non-positive degrees, so is each $K(C_p)$, and (b) follows from (a) and
Definition~\ref{def:cokoszul}. In (c), for each $p$ with $C_p\ne0$, the complex $K(C_p)$ is the minimal resolution
of the nonzero module $M_p$ by Lemma~\ref{lem:cokoszul}(a). Thus $\Lin(P_N)$ is a
direct sum of minimal linear resolutions of the modules $M_p\sh{-p}$. Its terms coincide
with those of $P_N$, so for $N\ne0$ the projective dimension is the maximum of the
lengths of these resolutions over $p$ with $C_p\ne0$. The final equality, when the
projective dimension is finite, follows from Lemma~\ref{lem:cokoszul}(c). The case
$N=0$ is immediate.
\end{proof}

In particular, $N$ is Koszul if and only if $H^p(\GG(N))=0$ for $p\neq0$ and
$H^0(\GG(N))$ is coKoszul; in this case $\GG(N)\cong H^0(\GG(N))$ in
$\Db(\Ld\GMod)$. Weakly Koszul modules correspond to the modules whose
``Koszul dual complex'' has colinear cohomology,
in accordance with \cite[Thm.~3.1]{Ma} and \cite[Thm.~4.4]{Ya1}.

\section{Applications}\label{sec:applications}
\subsection{Absolute Koszulity and co-coherence}\label{sec:A}

In this subsection we characterise absolutely Koszul algebras in terms of their Koszul
duals. Recall from Definition~\ref{def:cocoherent} that $\Ld$ is
graded left co-coherent if $\Ld\Fcp$ is abelian.
Our main result is the following.

\begin{theorem}\label{thm:A}
Let $\La$ be a finite-dimensional Koszul algebra. The following are equivalent.
\begin{enumerate}[label=\textup{(\arabic*)}]
\item $\La$ is absolutely Koszul.
\item Every object of $\Db(\La\gmod)$ has finite linearity defect.
\item For every $X\in\C^b(\Ld\Cop)$ and every $p$, one has $H^p(X)\in\Ld\Cop$.
\item $\Ld\Cop$ is abelian.
\item $\Ld$ is graded left co-coherent.
\end{enumerate}
In this case $\Ld\Cop=\Ld\Fcp$.
\end{theorem}

\begin{proof}
(1)$\Rightarrow$(2). Let $Y\in\Db(\La\gmod)$ with minimal projective resolution $C$
(Corollary~\ref{cor:linear-part}). Choose $a$ with $H^n(C)=0$ for $n<a$, and let
$Q=\Coker(d^{a-1}\colon C^{a-1}\to C^a)$, a finite-dimensional graded module. The complex
\begin{equation*}
\cdots\to C^{a-2}\to C^{a-1}\to C^a\to Q\to0
\end{equation*}
is exact and its differentials take values in the radicals. Hence $\sigma^{\le a}C$,
shifted so that $C^a$ sits in degree $0$, is the minimal projective resolution $P_Q$, and
\begin{equation*}
H^n(\Lin C)\cong H^{n-a}(\Lin P_Q)\qquad\text{for }n<a,
\end{equation*}
since the weight decompositions and the brutal truncation are compatible. As
$\ld(Q)<\infty$, this vanishes for $n<a-\ld(Q)$. Since $C$ is right bounded, $\Lin(C)$ has
bounded cohomology.

(2)$\Rightarrow$(3). Let $X\in\C^b(\Ld\Cop)$ and $Y=\FF(X)$. By
Corollary~\ref{cor:linear-part}, $\Lin(C_Y)\cong\bigoplus_pK(H^p(X))[-p]$, a finite direct
sum, which has bounded cohomology by (2). Hence each $K(H^p(X))$ has bounded cohomology.
As $H^p(X)$ is right bounded, $H^p(X)\in\Ld\GMod^{-,b}=\Ld\Cop$ by
Theorem~\ref{thm:Bo1}(a).

(3)$\Rightarrow$(4). If $f\colon A\to B$ is a morphism in $\Ld\Cop$, then $\Ker f$ and
$\Coker f$ are the cohomology modules of the two-term complex $A\xrightarrow{f}B$.

(4)$\Leftrightarrow$(5) is Lemma~\ref{lem:cocoherent}.

(5)$\Rightarrow$(1). Let $N\in\La\gmod$. By Theorem~\ref{thm:Bo1}(d), there is a bounded
complex $X$ in $\Ld\Inj$ with $\FF(X)\cong N$. Since $\Ld\Inj\subseteq\Ld\Fcp$ and $\Ld\Fcp$
is closed under kernels and cokernels, $H^p(X)\in\Ld\Fcp=\Ld\Cop$ by
Lemma~\ref{lem:cocoherent}. Thus each $K(H^p(X))$ has bounded cohomology, and so does
$\Lin(P_N)\cong\bigoplus_pK(H^p(X))[-p]$, a finite sum. Hence $\ld(N)<\infty$.

The last assertion is Lemma~\ref{lem:cocoherent}(b).
\end{proof}

Theorem~\ref{thm:A} refines \cite[Prop.~4.8]{Ya2} and removes the noetherianity
assumption of \cite[Thm.~4.5]{MZ}. Condition~(2) is the analogue for complexes of
absolute Koszulity; the proof of (1)$\Rightarrow$(2) shows that it reduces to the
leftmost cohomology module of the minimal projective resolution.

The statement of \cite[Thm.~5.8]{MZ} on the stable components of the Auslander--Reiten
quiver of a self-injective Koszul algebra of Loewy length greater than three remains
valid for absolutely Koszul algebras, with the same proof, since it only uses the
existence of weakly Koszul syzygies.

\subsection{The graded finitistic dimension}\label{sec:findim}

The classical finitistic dimension conjecture, formulated by Bass in 1960, asserts that
\begin{equation*}
\findim\La=\sup\{\pd_\La M\mid M\in\La\fmod,\ \pd_\La M<\infty\}
\end{equation*}
is finite for every finite-dimensional algebra $\La$. For a finite-dimensional graded algebra
$\La$ we consider its graded analogue
\begin{equation*}
\grfindim\La=\sup\{\pd M\mid M\in\La\gmod,\ \pd M<\infty\}.
\end{equation*}
The minimal graded projective resolution of a graded module is also a minimal projective
resolution of the underlying ungraded module, so $\grfindim\La\le\findim\La$. To the best
of our knowledge, $\grfindim$ has not been studied systematically.

The following reformulation of Corollary~\ref{cor:pd} expresses $\grfindim\La$ in terms
of the Koszul dual.

\begin{theorem}\label{thm:grfindim}
Let $\La$ be a finite-dimensional Koszul algebra. Then
\begin{equation*}
\grfindim\La=\sup\Big\{\CHL(\GG(N))\ \Big|\ 0\ne N\in\La\gmod,\ \dim_k
\textstyle\bigoplus_pH^p(\GG(N))<\infty\Big\}-1.
\end{equation*}
In particular, $\grfindim\La<\infty$ if and only if the graded cohomological lengths of the
complexes $\GG(N)$ with finite-dimensional total cohomology are uniformly bounded.
\end{theorem}

\begin{proof}
This is Corollary~\ref{cor:pd}.
\end{proof}

Let $c(\La)=\sup\{\pd M\mid M\in\mathcal K(\La),\ \pd M<\infty\}$ be the finitistic dimension
of the category of Koszul $\La$-modules. By Lemma~\ref{lem:cokoszul},
\begin{equation*}
c(\La)=\sup\{\lgr(C)-1\mid C\ \text{finite-dimensional coKoszul}\ \Ld\text{-module}\}.
\end{equation*}

\begin{corollary}\label{cor:A}
Let $\La$ be a finite-dimensional Koszul algebra with $d=\glld\La<\infty$; in particular $\La$
is absolutely Koszul. Then
\begin{equation*}
\grfindim\La\le d+c(\La).
\end{equation*}
By \cite[Thm.~4.5]{GMRSZ}, $c(\La)<\infty$; hence $\grfindim\La<\infty$.
\end{corollary}

\begin{proof}
Let $0\neq M\in\La\gmod$ with $\pd M<\infty$, and assume $\pd M>d$. By
Lemma~\ref{lem:syzygy}, $W=\Omega^dM\ne0$ is weakly Koszul, and $\pd W=\pd M-d<\infty$. By
Corollary~\ref{cor:weakly-koszul}(c),
$\pd W=\max_{p\colon C_p\ne0}\pd M_p$, where $C_p=H^p(\GG(W))\sh{-p}$ and
the nonzero $M_p$ are Koszul modules. Each $\pd M_p$ is finite because the linear
strands of the bounded complex $P_W$ are bounded. Hence
$\pd M_p\le c(\La)$ and $\pd M\le d+c(\La)$.
\end{proof}

Via Theorem~\ref{thm:grfindim}, Corollary~\ref{cor:A} says that the graded lengths of
finite-dimensional coKoszul $\Ld$-modules control those of all complexes $\GG(N)$ with
finite-dimensional total cohomology, up to the global linearity defect: when $\glld\La=d$,
each such complex has, after $d$ steps, colinear cohomology
(Corollary~\ref{cor:weakly-koszul}).

Motivated by Corollary~\ref{cor:A}, we formulate the following conjecture.

\begin{conjecture}\label{conj:findim}
Let $\La$ be a finite-dimensional graded algebra. If $\grfindim\La<\infty$, then
$\findim\La<\infty$.
\end{conjecture}

For finite-dimensional quadratic monomial algebras one has
$\findim\La\le\grfindim\La+2$ (Corollary~\ref{cor:findim-grfindim}).

\subsection{Refined Koszul dualities}\label{sec:refined}

In this subsection we show that, when $\La$ is absolutely Koszul, the coperfect category in
the dualities of \cite{Bo1} can be replaced by the abelian category $\Ld\Fcp$.

We recall the dualities of \cite{Bo1}, which hold without any co-coherence assumption on
$\Ld$.

\begin{theorem}[{\cite[Theorems~3.10, 3.19, 3.20, 4.4, 4.7 and 4.9]{Bo1}}]\label{thm:Bo1-dualities}
Let $\La$ be a finite-dimensional Koszul algebra.
\begin{enumerate}[label=\textup{(\roman*)}]
\item $\FF\colon\Db(\Ld\Cop)\xrightarrow{\sim}\Db(\La\gmod)$.
\item $\Db(\Ld\Cop/\Ld\gmod)\xrightarrow{\sim}\Dsg(\La\gmod)$; if $\La$ is
Iwanaga--Gorenstein, the latter category is also equivalent to $\La\Gproj^{\mathbb Z}$.
\item There are triangulated equivalences
\begin{gather*}
\mathrm H^0\big(\pretr(\mathsf D^b_{\mathrm{dg}}(\Ld\Cop)/\sh1[-1])\big)\xrightarrow{\sim}\Db(\La\fmod),\\
\mathrm H^0\big(\pretr(\mathsf D^b_{\mathrm{dg}}(\Ld\Cop/\Ld\gmod)/\sh1[-1])\big)\xrightarrow{\sim}\Dsg(\La\fmod);
\end{gather*}
if $\La$ is Iwanaga--Gorenstein, the latter category is also equivalent to $\La\Gproj$.
\end{enumerate}
Here $\mathrm H^0(\pretr(-))$ denotes the homotopy category of the pretriangulated hull of
the corresponding dg orbit category; see \cite[\S2.4]{Bo1}. The equivalence in (ii) is
induced by $\FF$, compatibly with the quotient functors.
\end{theorem}

As a consequence of Theorem~\ref{thm:A}, we obtain the following refinement of
Theorem~\ref{thm:Bo1-dualities}.

\begin{corollary}\label{cor:B}
Let $\La$ be a finite-dimensional absolutely Koszul algebra. Then $\Ld\Fcp=\Ld\Cop$ is an
abelian subcategory of $\Ld\GMod$ with enough injectives and finite global dimension,
$\Ld\gmod$ is a Serre subcategory of it, and there are triangulated equivalences
\begin{enumerate}[label=\textup{(\roman*)}]
\item $\Db\big(\Ld\Fcp\big)\xrightarrow{\sim}\Db\big(\La\gmod\big)$;
\item $\Db\big(\Ld\Fcp/\Ld\gmod\big)\xrightarrow{\sim}\Dsg\big(\La\gmod\big)$;
\item $\Db\big(\Ld\Fcp/\Ld\gmod\big)\xrightarrow{\sim}\La\Gproj^{\mathbb Z}$ if $\La$ is Iwanaga--Gorenstein;
\item $\mathrm H^0\big(\pretr\big(\mathsf D^b_{\mathrm{dg}}(\Ld\Fcp)/\sh{1}[-1]\big)\big)\xrightarrow{\sim}\Db(\La\fmod)$;
\item $\mathrm H^0\big(\pretr\big(\mathsf D^b_{\mathrm{dg}}(\Ld\Fcp/\Ld\gmod)/\sh{1}[-1]\big)\big)\xrightarrow{\sim}\Dsg(\La\fmod)$;
\item $\mathrm H^0\big(\pretr\big(\mathsf D^b_{\mathrm{dg}}(\Ld\Fcp/\Ld\gmod)/\sh{1}[-1]\big)\big)\xrightarrow{\sim}\La\Gproj$ if $\La$ is Iwanaga--Gorenstein.
\end{enumerate}
In particular, $\Ld\Fcp/\Ld\gmod$ is an abelian category whose bounded derived category is
equivalent to $\Dsg(\La\gmod)$.
\end{corollary}

\begin{proof}
By Theorem~\ref{thm:A}, $\Ld\Fcp=\Ld\Cop$ and it is closed under kernels and cokernels.
Its exact structure as in \cite{Bo1} consists of all short exact sequences of $\Ld\GMod^-$
with terms in $\Ld\Cop$, which is the abelian exact structure of $\Ld\Fcp$. The objects of
$\Ld\Inj$ are injective, every object embeds into one of them, and coresolutions have
length at most $\ell$ (Lemma~\ref{lem:ext-vanishing}). Since $\Ld\gmod$ is a Serre
subcategory of $\Ld\GMod$ contained in $\Ld\Cop$, the quotient $\Ld\Fcp/\Ld\gmod$ is an
abelian Serre quotient. The equivalences now follow by substituting in
Theorem~\ref{thm:Bo1-dualities}.
\end{proof}

\subsection{Koszul \texorpdfstring{$t$}{t}-structures}\label{sec:t-structures}

In this subsection we use the refined dualities of Subsection~\ref{sec:refined} to construct
$t$-structures on $\Db(\La\gmod)$ and on $\Dsg(\La\gmod)$, and we describe their aisles
and hearts in terms of minimal projective resolutions.

Throughout this subsection, $\La$ is a finite-dimensional absolutely Koszul algebra. By
Theorem~\ref{thm:A}, $\Ld\Fcp=\Ld\Cop$ is abelian, hence $\Db(\Ld\Fcp)$ carries the standard
$t$-structure $(\mathsf D^{\le0},\mathsf D^{\ge0})$, given by the vanishing of cohomology in
positive, respectively negative, degrees, with heart $\Ld\Fcp$. We transport it along the
equivalences of Corollary~\ref{cor:B}. For $Y\in\Db(\La\gmod)$ we denote by $C_Y$ its minimal
projective resolution (Definition~\ref{def:ld}).

We first show that the cohomology of an object of $\Db(\Ld\Fcp)$ is detected by
morphisms to the injective modules $I^!_x\sh m$.

\begin{lemma}\label{lem:Hom-injective}
Let $X\in\Db(\Ld\Fcp)$, $x\in Q_0$ and $m,i\in\mathbb Z$. Then
\begin{equation*}
\Hom_{\Db(\Ld\Fcp)}\big(X,I^!_x\sh m[-i]\big)\cong\Hom_{\Ld\Fcp}\big(H^i(X),I^!_x\sh m\big).
\end{equation*}
In particular, $H^i(X)=0$ if and only if $\Hom_{\Db(\Ld\Fcp)}(X,I^!_x\sh m[-i])=0$ for all
$x\in Q_0$ and $m\in\mathbb Z$.
\end{lemma}

\begin{proof}
Put $I=I^!_x\sh m$, an injective object of $\Ld\Fcp$. Morphisms $X\to I[-i]$ in
$\Db(\Ld\Fcp)$ are therefore computed in the homotopy category. A chain map $X\to I[-i]$ is
a morphism $X^i/\Img d^{i-1}_X\to I$, and it is null-homotopic if and only if it factors
through the induced map $X^i/\Img d^{i-1}_X\to\Img d^i_X$. Applying the exact functor
$\Hom(-,I)$ to $0\to H^i(X)\to X^i/\Img d^{i-1}_X\to\Img d^i_X\to0$, and using that every
map $\Img d^i_X\to I$ extends to $X^{i+1}$, gives the first claim. For the second, a nonzero
module $M\in\Ld\Fcp$ is right bounded, hence has a simple submodule $S^!_x\sh m$; the
inclusion $S^!_x\sh m\to I^!_x\sh m$ extends to a nonzero map $M\to I^!_x\sh m$.
\end{proof}

We can now describe the transported $t$-structure on $\Db(\La\gmod)$ intrinsically.

\begin{theorem}\label{thm:t-derived}
Let $\La$ be a finite-dimensional absolutely Koszul algebra, and put
\begin{equation*}
\mathsf D^{\le0}_{\FF}:=\FF\big(\mathsf D^{\le0}(\Ld\Fcp)\big),\qquad
\mathsf D^{\ge0}_{\FF}:=\FF\big(\mathsf D^{\ge0}(\Ld\Fcp)\big).
\end{equation*}
\begin{enumerate}[label=\textup{(\alph*)}]
\item $(\mathsf D^{\le0}_{\FF},\mathsf D^{\ge0}_{\FF})$ is a $t$-structure on $\Db(\La\gmod)$
with heart $\mathcal H_{\FF}:=\FF(\Ld\Fcp)\simeq\Ld\Fcp$, and its cohomology functors are
$\mathrm H^p_{\FF}(Y)\cong\FF\big(H^p(\GG(Y))\big)$.
\item For $Y\in\Db(\La\gmod)$ the following are equivalent:
\begin{enumerate}[label=\textup{(\roman*)}]
\item $Y\in\mathsf D^{\le0}_{\FF}$;
\item $\Hom_{\Db(\La\gmod)}\big(Y,S_x\sh m[-m-n]\big)=0$ for all $n>0$, $x\in Q_0$ and
$m\in\mathbb Z$;
\item the linear strands $\Lin(C_Y)_s$ vanish for $s<0$;
\item for every $n$, the term $C_Y^n$ is generated in degrees $\le-n$.
\end{enumerate}
\item The same holds for $\mathsf D^{\ge0}_{\FF}$, with $n<0$ in \textup{(ii)}, $s>0$ in
\textup{(iii)} and ``$\ge-n$'' in \textup{(iv)}.
\item $Y\in\mathcal H_{\FF}$ if and only if $C_Y$ is linear. In particular, a module
$N\in\La\gmod$ lies in $\mathcal H_{\FF}$ if and only if $N$ is Koszul.
\item For every $Y\in\Db(\La\gmod)$ there is an isomorphism of complexes
\begin{equation*}
\Lin(C_Y)\ \cong\ \bigoplus_{p\in\mathbb Z}C_{\mathrm H^p_{\FF}(Y)}[-p],
\end{equation*}
under which the linear strand of weight $s$ of $C_Y$ is $C_{\mathrm H^{-s}_{\FF}(Y)}[s]$.
\end{enumerate}
\end{theorem}

\begin{proof}
(a) $\FF$ is a triangulated equivalence (Corollary~\ref{cor:B}), so it transports the
standard $t$-structure, its heart and its cohomology functors.

(b) Put $X=\GG(Y)$, so that (i) means $H^p(X)=0$ for $p>0$. By Theorem~\ref{thm:Bo1}(c),
$\FF(I^!_x\sh{m'}[-n])\cong S_x\sh{-m'}[m'-n]$; with $m=-m'$, Lemma~\ref{lem:Hom-injective}
gives (i)$\Leftrightarrow$(ii). By Corollary~\ref{cor:linear-part} and
Theorem~\ref{thm:structure}(c), the linear strand of weight $s$ of $C_Y$ is
$K(H^{-s}(X))[s]$, and $K(M)=0$ if and only if $M=0$; this gives (i)$\Leftrightarrow$(iii).
Finally $\Lin(C_Y)^n_s=(C_Y^n)_{(n+s)}$ and $P_x\sh a$ is generated in degree $-a$, so (iii)
says that $C_Y^n\in\add\{P_x\sh a\mid a\ge n\}$, which is (iv).

(c) The proof is the same.

(d) By (b) and (c), $Y\in\mathcal H_{\FF}$ if and only if $C_Y^n\in\add\{P_x\sh n\mid x\in Q_0\}$
for all $n$, that is, $C_Y$ is linear. For a module $N$ we have $C_N=P_N$, which is linear
exactly when $N$ is Koszul.

(e) By (a), $\mathrm H^p_{\FF}(Y)\cong\TT\big(K(H^p(X))\big)$, and the linear complex
$K(H^p(X))$ is minimal, hence it is the minimal projective resolution of $\mathrm H^p_{\FF}(Y)$. The claim
is now Corollary~\ref{cor:linear-part} and Theorem~\ref{thm:structure}(c).
\end{proof}

We now turn to the singularity category.

Let $\pi\colon\Db(\La\gmod)\to\Dsg(\La\gmod)$ and
$q\colon\Db(\Ld\Fcp)\to\Db(\Ld\Fcp/\Ld\gmod)$ be the quotient functors, and let
\begin{equation*}
\Phi\colon\Db\big(\Ld\Fcp/\Ld\gmod\big)\xrightarrow{\ \sim\ }\Dsg(\La\gmod)
\end{equation*}
be the equivalence of Corollary~\ref{cor:B}(ii), so that $\Phi\circ q\cong\pi\circ\FF$
(Theorem~\ref{thm:Bo1-dualities}). We put
$\mathsf D^{\le0}_{\Phi}:=\Phi(\mathsf D^{\le0})$ and $\mathsf D^{\ge0}_{\Phi}:=\Phi(\mathsf D^{\ge0})$,
where $(\mathsf D^{\le0},\mathsf D^{\ge0})$ is the standard $t$-structure on
$\Db(\Ld\Fcp/\Ld\gmod)$.

\begin{theorem}\label{thm:t-singular}
Let $\La$ be a finite-dimensional absolutely Koszul algebra.
\begin{enumerate}[label=\textup{(\alph*)}]
\item $\mathsf D^{\le0}_{\Phi}$ consists of the objects isomorphic to $\pi(Y)$ with
$Y\in\mathsf D^{\le0}_{\FF}$, that is, with $C_Y^n$ generated in degrees $\le-n$ for all $n$;
similarly for $\mathsf D^{\ge0}_{\Phi}$. In particular, $\pi$ is $t$-exact.
\item The heart $\mathcal H_\Phi$ is equivalent to $\Ld\Fcp/\Ld\gmod$, and its objects are,
up to isomorphism, the objects $\pi(L[g])$, where $g\in\mathbb Z$ and $L\in\La\gmod$ is such
that $L\sh g$ is Koszul.
\item If $\La$ is Iwanaga--Gorenstein, the equivalence of Corollary~\ref{cor:B}(iii)
transports the standard $t$-structure to a $t$-structure on $\La\Gproj^{\mathbb Z}$ whose heart
is equivalent to $\Ld\Fcp/\Ld\gmod$.
\end{enumerate}
\end{theorem}

\begin{proof}
(a) The functor $q$ is induced by an exact functor of abelian categories, so it commutes with
cohomology and is $t$-exact. Hence $\pi(\mathsf D^{\le0}_{\FF})=\Phi q(\mathsf D^{\le0}(\Ld\Fcp))
\subseteq\mathsf D^{\le0}_{\Phi}$. Conversely, let $Z\in\mathsf D^{\le0}_{\Phi}$. Since $\pi$ is
the identity on objects, $Z=\pi(\FF X)\cong\Phi(qX)$ for some $X\in\Db(\Ld\Fcp)$, and
$qX\in\mathsf D^{\le0}$. Thus $q(H^p(X))=H^p(qX)=0$, that is $H^p(X)\in\Ld\gmod$, for $p>0$.
Consider the truncation triangle $\tau^{\le0}X\to X\to\tau^{\ge1}X\to$. All cohomology of
$\tau^{\ge1}X$ lies in $\Ld\gmod$, so $q(\tau^{\ge1}X)=0$ and $\pi\FF(\tau^{\ge1}X)=0$. Hence
$Z\cong\pi\FF(\tau^{\le0}X)$ with $\FF(\tau^{\le0}X)\in\mathsf D^{\le0}_{\FF}$. The description
in terms of minimal projective resolutions is Theorem~\ref{thm:t-derived}(b). The case of
$\mathsf D^{\ge0}_{\Phi}$ is identical.

(b) The first assertion is clear. Every object of $\Ld\Fcp/\Ld\gmod$ is the image of some
$M\in\Ld\Fcp$, and $\Phi(qM)\cong\pi\FF(M)$. Since $K(M)$ has bounded cohomology, there is
$r\in\mathbb Z$ with $H^n(K(M))=0$ for $n<r$. The submodule $M_{>r}=\bigoplus_{n>r}M_n$ is
finite-dimensional, so $q(M)\cong q(M_{\le r})$, where $M_{\le r}=M/M_{>r}$. Since $K(M)^n$ and
$d^n$ only involve $M_n$ and $M_{n+1}$, we have $K(M_{\le r})=\sigma^{\le r}K(M)$, a linear
complex exact in degrees $<r$. It is therefore a linear resolution, placed in degree $r$, of
$L:=\Coker(d^{r-1})$, which is generated in degree $-r$. With $g=-r$, the module $L\sh g$ is
Koszul and $\FF(M_{\le r})\cong L[g]$. Hence $\Phi(qM)\cong\pi(L[g])$. Conversely, if
$L\sh g$ is Koszul, then $L\sh g\in\mathcal H_{\FF}$ by Theorem~\ref{thm:t-derived}(d), so
$L[g]\cong(L\sh g)\sh{-g}[g]\in\mathcal H_{\FF}$, since $\mathcal H_{\FF}$ is stable under
$\sh{1}[-1]$ and its inverse by \eqref{eq:K-shift}, and
$\pi(L[g])\in\mathcal H_\Phi$ by (a).

(c) This is clear.
\end{proof}

\subsection{Modules with linear presentations}\label{sec:linear-pres}

In this subsection we give a necessary and sufficient condition, in terms of the Koszul dual,
for every module with a linear presentation to be Koszul. Unlike in the previous subsections,
no absolute Koszulity assumption is needed.

A module $M\in\La\gmod$ has a \emph{linear presentation} if there is an exact sequence
$P^{-1}\to P^0\to M\to0$ with $P^0\in\add\{P_x\sh0\}$ and $P^{-1}\in\add\{P_x\sh{-1}\}$, i.e.\
$P^0$ is generated in degree $0$ and $P^{-1}$ in degree $1$. Let $\mathcal L(\La)$ be the
category of such modules; clearly $\mathcal K(\La)\subseteq\mathcal L(\La)$. Green et al.
\cite{GMRSZ} observed that a necessary and sufficient condition on $\La$ ensuring
$\mathcal L(\La)=\mathcal K(\La)$ was not known; see also \cite{AG,GMZ}.

Let $N\in\Ld\gmod$ be concentrated in degrees $-1$ and $0$. Then $K(N)$ is the two-term
linear complex
\begin{equation*}
K(N)\colon\quad \bigoplus_xP_x\sh{-1}\otimes N_{-1}(x)\xrightarrow{\ d\ }
\bigoplus_yP_y\sh0\otimes N_0(y),
\end{equation*}
in degrees $-1$ and $0$. Note that $N_0\subseteq\soc N$, since $N_1=0$.

The following lemma will be used in the proof of Theorem~\ref{thm:C}.

\begin{lemma}[Second syzygy decomposition]\label{lem:second-syzygy}
Let $N\in\Ld\gmod$ be concentrated in degrees $-1$ and $0$, and set $M=H^0(K(N))$.
\begin{enumerate}[label=\textup{(\alph*)}]
\item $M\in\mathcal L(\La)$, and $K(N)$ is a linear presentation of $M$. The degree-one part
of $\Ker d$ is isomorphic to $\soc(N)_{-1}$; in particular $K(N)$ is a minimal presentation
of $M$ if and only if $\soc N=N_0$.
\item There is a decomposition $N=N'\oplus N''$, where $N''=\soc(N)_{-1}$ is semisimple and
concentrated in degree $-1$, and $\soc N'=N'_0$. Accordingly,
\begin{equation*}
H^{-1}(K(N))\ \cong\ \Omega^2_\La(M)\ \oplus\ \bigoplus_{x\in Q_0}P_x\sh{-1}\otimes N''(x).
\end{equation*}
\item Every $M\in\mathcal L(\La)$ has a minimal presentation of the form $K(N)$ with $N$ as
above and $\soc N=N_0$; such an $N$ is unique up to isomorphism.
\end{enumerate}
\end{lemma}

\begin{proof}
(a) The first assertion is clear. The differential $d$ has image in $\rad\big(\bigoplus_y
P_y\sh 0\otimes N_0(y)\big)$, so the second term maps onto $M$ by a projective cover. The
term of degree $-1$ is generated in degree $1$, and its degree-one part is
$\bigoplus_x ke_x\otimes N_{-1}(x)$. For $u_x\in N_{-1}(x)$,
\begin{equation*}
d\Big(\sum_xe_x\otimes u_x\Big)=\sum_x\sum_{\alpha\colon y\to x}\alpha\otimes N(\alpha^{\mathrm{op}})(u_x).
\end{equation*}
Distinct arrows are linearly independent in $\La_1=kQ_1$, so this vanishes if and only if
$N(\alpha^{\mathrm{op}})(u_x)=0$ for all arrows $\alpha$ ending at $x$, that is
$u_x\in\soc(N)_{-1}(x)$ \cite[\S3.2]{Bo1}. Since the degree $-1$ term is generated in degree
one, $\Ker d$ lies in its radical if and only if $(\Ker d)_1=0$, i.e.\ $\soc(N)_{-1}=0$; this is
exactly the condition for $\bigoplus_x P_x\sh{-1}\otimes N_{-1}(x)\to\Img d=\Omega_\La(M)$ to be a
projective cover.

(b) $N''=\soc(N)_{-1}$ is a submodule. Choose vertexwise complements $V(x)$ of $N''(x)$ in
$N_{-1}(x)$; then $N'=V\oplus N_0$ is a submodule, since arrows map degree $-1$ to degree
$0$ and degree $0$ to zero, and $N=N'\oplus N''$ with $\soc(N')_{-1}=0$. Now
$K(N)=K(N')\oplus K(N'')$, where $K(N'')$ is $\bigoplus_xP_x\sh{-1}\otimes N''(x)$ in degree
$-1$ with zero differential. By (a), $K(N')$ is a minimal presentation of
$H^0(K(N'))=H^0(K(N))=M$, so $H^{-1}(K(N'))$ is the kernel of the projective cover of
$\Omega_\La(M)$, that is $\Omega^2_\La(M)$.

(c) If $M$ has a linear presentation, its minimal presentation is a direct summand of it,
hence again linear. A two-term linear complex with finitely generated terms is $K(N)$ for a
finite-dimensional $N$ concentrated in degrees $-1,0$, unique up to isomorphism since $K$ is
fully faithful, and $\soc N=N_0$ by (a).
\end{proof}

For a module $C\in\Ld\GMod^-$ with $C_n=0$ for $n>0$, the subspace $C_{-1}\oplus C_0$ is a
submodule, which we denote by $C_{\ge-1}$. Clearly $K(C_{\ge-1})=\sigma^{\ge-1}K(C)$.

We can now state the main result of this subsection.

\begin{theorem}\label{thm:C}
Let $\La$ be a finite-dimensional Koszul algebra. The following are equivalent.
\begin{enumerate}[label=\textup{(\alph*)}]
\item $\mathcal L(\La)=\mathcal K(\La)$.
\item Every finite-dimensional graded $\Ld$-module $N$ concentrated in degrees $-1$ and $0$
with $\soc N=N_0$ is isomorphic to $C_{\ge-1}$ for some coKoszul module $C$.
\end{enumerate}
\end{theorem}

\begin{proof}
(a)$\Rightarrow$(b). Let $N$ be as in (b) and $M=H^0(K(N))$. By
Lemma~\ref{lem:second-syzygy}(a), $M\in\mathcal L(\La)=\mathcal K(\La)$. If $M=0$, then
$K(N)$ is a minimal presentation of zero, so $N=0$ and we take $C=0$. Otherwise, by
Lemma~\ref{lem:cokoszul}(b), $P_M\cong K(C)$ for a coKoszul module $C$. The brutal truncation
$\sigma^{\ge-1}P_M$ is a minimal presentation of $M$, as is $K(N)$ by
Lemma~\ref{lem:second-syzygy}(a). Minimal presentations are unique up to isomorphism of
complexes, so
\begin{equation*}
K(N)\cong\sigma^{\ge-1}K(C)=K(C_{\ge-1}),
\end{equation*}
and $N\cong C_{\ge-1}$ since $K$ is fully faithful.

(b)$\Rightarrow$(a). Let $0\ne M\in\mathcal L(\La)$. By Lemma~\ref{lem:second-syzygy}(c), $M$
has a minimal presentation $K(N)$ with $\soc N=N_0$, and by (b), $N\cong C_{\ge-1}$ with $C$
coKoszul. By Lemma~\ref{lem:cokoszul}(a), $K(C)$ is the minimal resolution of the Koszul
module $H^0(K(C))$, and
\begin{equation*}
H^0(K(C))=H^0\big(\sigma^{\ge-1}K(C)\big)=H^0\big(K(C_{\ge-1})\big)\cong H^0(K(N))=M.
\end{equation*}
Hence $M$ is Koszul, so $\mathcal L(\La)\subseteq\mathcal K(\La)$; the reverse inclusion is
clear.
\end{proof}

\section{Finite-dimensional quadratic monomial algebras}\label{sec:monomial}

In this section we show that finite-dimensional quadratic monomial algebras are
absolutely Koszul and we describe the objects studied in the previous sections for this
class of algebras. The main ingredients are a result of Zimmermann Huisgen on second
syzygies over monomial algebras \cite{ZH} and the modules $\Ia{\alpha}$ introduced in
Definition~\ref{def:alphaI}.

\subsection{Notation and terminology}
Paths are composed from right to left: if $\alpha\colon x\to y$ and $\beta\colon y\to z$ are
arrows, then $\beta\alpha$ is the path from $x$ to $z$ which first traverses $\alpha$. Let
$\La=kQ/(R)$, where $R$ is a set of paths of length two and $(R)$ is admissible, so that
$\La$ is finite-dimensional. The nonzero paths form a basis of $\La$, and a path is nonzero
if and only if none of its subpaths of length two lies in $R$. Hence, if $u$ and $q$ are
nonzero composable paths of positive length, then
\begin{equation}\label{eq:monomial-product}
uq\neq 0 \iff \delta\lambda\notin R,
\end{equation}
where $\delta$ is the first arrow traversed by $u$ and $\lambda$ is the last arrow traversed
by $q$. The algebra $\La$ is Koszul \cite{GZ}, and
\begin{equation*}
\Ld=kQ^{\mathrm{op}}\big/\big(\alpha^{\mathrm{op}}\beta^{\mathrm{op}}\ \big|\ \beta\alpha
\text{ a path of length two with } \beta\alpha\notin R\big),
\end{equation*}
where $(\beta\alpha)^{\mathrm{op}}=\alpha^{\mathrm{op}}\beta^{\mathrm{op}}$. Thus $\Ld$ is again
quadratic monomial, and a path $\alpha_1^{\mathrm{op}}\cdots\alpha_n^{\mathrm{op}}$ of
$Q^{\mathrm{op}}$ is nonzero in $\Ld$ if and only if
\begin{equation}\label{eq:R-chains}
\alpha_{j+1}\alpha_j\in R\qquad\text{for } 1\le j<n .
\end{equation}
The module $I^!_x(z)=\DD\Ld(z,x)$ has the basis $\{q^*\}$ dual to the nonzero paths $q$ of
$\Ld$ from $z$ to $x$, with $q^*$ in degree $-\ell(q)$. If $\beta\colon y\to z$ is an arrow of
$Q$, then $I^!_x(\beta^{\mathrm{op}})$ sends $q^*$ to $p^*$ if $q=p\beta^{\mathrm{op}}$, and to
zero otherwise; that is, it deletes the first arrow traversed by $q$ when this arrow is
$\beta^{\mathrm{op}}$.

\subsection{The modules \texorpdfstring{$\La\alpha$}{Lambda alpha} and
\texorpdfstring{$\Ia{\alpha}$}{I alpha}}

For an arrow $\alpha\colon x\to y$ of $Q$, let $\La\alpha\subseteq P_x$ be the submodule
generated by $\alpha$, that is, the image of $P_\alpha\colon P_y\to P_x$. It is generated in
degree $1$, with basis the nonzero paths $u\alpha$. By \eqref{eq:monomial-product},
\begin{equation}\label{eq:syzygy-Lalpha}
\Ker\big(P_y\sh{-1}\twoheadrightarrow\La\alpha\big)=\bigoplus_{\beta\colon y\to w,\ \beta\alpha\in R}\La\beta\sh{-1},
\qquad e_y\mapsto\alpha,
\end{equation}
so $\Omega(\La\alpha\sh c)\cong\bigoplus_{\beta\alpha\in R}\La\beta\sh{c-1}$. Similarly, for a
path $p$ of positive length with last arrow $\lambda$,
\begin{equation}\label{eq:Lp}
\La p\cong\La\lambda\sh{1-\ell(p)},\qquad up\mapsto u\lambda .
\end{equation}

\begin{definition}\label{def:alphaI}
Let $\alpha\colon x\to y$ be an arrow of $Q$, and let $C_\alpha\subseteq I^!_x$ be the span of
the elements $q^*$ where $q$ runs through the nonzero paths of $\Ld$ ending at $x$ whose last
arrow is not $\alpha^{\mathrm{op}}$ (including $e_x$). Put $\Ia{\alpha}:=I^!_x/C_\alpha$.
\end{definition}

These modules have the following properties.

\begin{lemma}\label{lem:alphaI}
Let $\alpha\colon x\to y$ be an arrow of $Q$.
\begin{enumerate}[label=\textup{(\alph*)}]
\item $C_\alpha$ is a submodule of $I^!_x$, and $\Ia{\alpha}$ has the basis given by the classes
of the elements $(\alpha^{\mathrm{op}}q')^*$, where $\alpha^{\mathrm{op}}q'\neq0$; the element
$(\alpha^{\mathrm{op}}q')^*$ has degree $-1-\ell(q')$, and arrows act by deleting the first
arrow of $q'$.
\item $\soc\Ia{\alpha}$ is spanned by the class of $(\alpha^{\mathrm{op}})^*$, and is isomorphic
to $S^!_y\sh1$.
\item $\Ia{\alpha}$ is isomorphic to the submodule of $I^!_y\sh1$ spanned by the elements
$q'^*$ with $\alpha^{\mathrm{op}}q'\neq0$.
\item There is an exact sequence
\begin{equation*}
0\longrightarrow S^!_y\sh1\longrightarrow\Ia{\alpha}\longrightarrow
\bigoplus_{\gamma\colon y\to w,\ \gamma\alpha\in R}\Ia{\gamma}\sh1\longrightarrow0 .
\end{equation*}
\end{enumerate}
\end{lemma}

\begin{proof}
(a) Arrows delete the first arrow traversed; this preserves the last arrow of a path of
length at least two, and sends paths of length one to $e_x$ or to zero. (b) is clear from
(a). (c) The bijection $(\alpha^{\mathrm{op}}q')^*\mapsto q'^*$ shifts degrees by one and
commutes with the action, since both sides delete the first arrow of $q'$; its image is a
submodule, since subpaths of nonzero paths are nonzero. (d) By \eqref{eq:R-chains}, a path
$\alpha^{\mathrm{op}}q'$ with $\ell(q')\ge1$ is nonzero if and only if
$q'=\gamma^{\mathrm{op}}q$ with $\gamma\alpha\in R$ and $\gamma^{\mathrm{op}}q\neq0$. The map
$(\alpha^{\mathrm{op}}\gamma^{\mathrm{op}}q)^*\mapsto(\gamma^{\mathrm{op}}q)^*$,
$(\alpha^{\mathrm{op}})^*\mapsto0$, is a surjective morphism with kernel the socle.
\end{proof}

The next proposition shows that $\Ia{\alpha}$ corresponds to the module $\La\alpha$
under the Koszul functor.

\begin{proposition}\label{prop:K-alphaI}
Let $\alpha\colon x\to y$ be an arrow of $Q$. Then $K(\Ia{\alpha})$ vanishes in degrees $\ge0$,
is exact in degrees $\le-2$, and $H^{-1}(K(\Ia{\alpha}))\cong\La\alpha$. Hence
$K(\Ia{\alpha})$ is the minimal projective resolution of $\La\alpha$, placed in degree $-1$,
$\La\alpha\sh1$ is Koszul, $\Ia{\alpha}\sh{-1}$ is coKoszul, $\Ia{\alpha}\in\Ld\Cop$ and
$\FF(\Ia{\alpha})\cong\La\alpha[1]$.
\end{proposition}

\begin{proof}
For $n\ge1$, the term $K(\Ia{\alpha})^{-n}$ has the basis $u\otimes(\alpha^{\mathrm{op}}q')^*$,
where $q'$ is a path of $\Ld$ from $z$ to $y$ of length $n-1$ with
$\alpha^{\mathrm{op}}q'\neq0$, and $u$ is a nonzero path of $\La$ starting at $z$. If $n\ge2$,
write $q'=p'\gamma^{\mathrm{op}}$ with $\gamma\colon w\to z$; by Lemma~\ref{lem:alphaI}(a),
\begin{equation*}
d\big(u\otimes(\alpha^{\mathrm{op}}q')^*\big)=u\gamma\otimes(\alpha^{\mathrm{op}}p')^*,
\end{equation*}
while $d$ vanishes on $K(\Ia{\alpha})^{-1}$. Since $\gamma$ is the first arrow traversed by
$u\gamma$, the assignment $(u,q')\mapsto(u\gamma,p')$ is injective on the pairs with
$u\gamma\neq0$. Hence $\Ker d$ is spanned by the basis elements it contains.

Let $n\ge2$ and $u\otimes(\alpha^{\mathrm{op}}q')^*$ with $u\gamma=0$. Then $u=u'\beta$ has
positive length, with first arrow $\beta\colon z\to v$, and $\beta\gamma\in R$ by
\eqref{eq:monomial-product}. By \eqref{eq:R-chains}, $\alpha^{\mathrm{op}}q'\beta^{\mathrm{op}}\neq0$,
and $d\big(u'\otimes(\alpha^{\mathrm{op}}q'\beta^{\mathrm{op}})^*\big)=u\otimes(\alpha^{\mathrm{op}}q')^*$.
Thus $K(\Ia{\alpha})$ is exact in degrees $\le-2$. In degree $-1$ we have
$K(\Ia{\alpha})^{-1}=P_y\sh{-1}\otimes k(\alpha^{\mathrm{op}})^*$, and the image of $d^{-2}$ is
spanned by the elements $u\beta\otimes(\alpha^{\mathrm{op}})^*$ with $\beta\colon y\to w$ and
$\beta\alpha\in R$. By \eqref{eq:syzygy-Lalpha}, $H^{-1}(K(\Ia{\alpha}))\cong\La\alpha$.

The term in degree $-1-n$ is generated in degree $1+n$, so $\La\alpha\sh1$ is Koszul. By
\eqref{eq:K-shift}, $K(\Ia{\alpha}\sh{-1})\cong K(\Ia{\alpha})\sh1[-1]$ has cohomology only in
degree $0$, so $\Ia{\alpha}\sh{-1}$ is coKoszul. Finally $\Ia{\alpha}\in\Ld\Cop$ by
Theorem~\ref{thm:Bo1}(a), and $\FF(\Ia{\alpha})=\TT K(\Ia{\alpha})\cong\La\alpha[1]$.
\end{proof}

\subsection{Second syzygies and absolute Koszulity}

We shall use the following lemma several times.

\begin{lemma}\label{lem:single-degree}
Let $P\in\La\proj$, and let $W\subseteq P$ be a graded submodule generated by $W_d$, with
$W_n=0$ for $n<d$. Let $\pi\colon\bigoplus_yP_y\sh{-d}\otimes W_d(y)\to W$ be the projective
cover. Then
\begin{equation*}
\Ker\pi\cong\bigoplus_{\delta\in Q_1}\La\delta\sh{-d}\otimes K_\delta
\end{equation*}
for finite-dimensional vector spaces $K_\delta$. In particular $W\sh d$ is Koszul.
\end{lemma}

\begin{proof}
Write $P=\bigoplus_kP_{x_k}\sh{n_k}$, and let $\varepsilon_k$ be the generator of the summand
$P_{x_k}\sh{n_k}$. Since the nonzero paths form a basis of $\La$, the elements $q\varepsilon_k$,
where $q$ is a nonzero path starting at $x_k$, form a basis of $P$. Hence, for $y\in Q_0$, every
$v\in W_d(y)$ can be written uniquely as
\begin{equation*}
v=\sum_k\sum_qc_{k,q}(v)\,q\varepsilon_k,\qquad c_{k,q}(v)\in k,
\end{equation*}
where $q$ runs through the nonzero paths from $x_k$ to $y$ of length $d+n_k$.

Put $G=\bigoplus_{y\in Q_0}P_y\sh{-d}\otimes W_d(y)$, so that $\pi(e_y\otimes v)=v$. Every element
of $G$ can be written uniquely as $\xi=\sum_uu\otimes v_u$, where $u$ runs through the nonzero
paths of $\La$ and $v_u\in W_d(s(u))$, and
\begin{equation*}
\pi(\xi)=\sum_u\sum_{k,q}c_{k,q}(v_u)\,(uq)\,\varepsilon_k .
\end{equation*}
If $uq\varepsilon_k=u'q'\varepsilon_{k'}$ are nonzero, then $k=k'$, and since $q$ and $q'$ have the
same length $d+n_k$, we get $q=q'$ and $u=u'$. Thus the nonzero elements $uq\varepsilon_k$ occurring
in $\pi(\xi)$ are pairwise distinct elements of a basis of $P$, and $\xi\in\Ker\pi$ if and only if
$v_u\in K(u)$ for every $u$, where
\begin{equation*}
K(u)=\{v\in W_d(s(u))\mid c_{k,q}(v)=0\ \text{whenever}\ uq\neq0\}.
\end{equation*}
Clearly $K(e_y)=0$ for $y\in Q_0$. If $u$ has positive length and first arrow $\delta$, then,
for every nonzero path $q$ ending at $s(u)$, the condition $uq\neq0$ only depends on $\delta$, by
\eqref{eq:monomial-product}; hence $K(u)=:K_\delta$ only depends on $\delta$, and
$K_\delta\subseteq W_d(s(\delta))$. Consequently
\begin{equation*}
\Ker\pi=\bigoplus_{\delta\in Q_1}U_\delta,\qquad
U_\delta=\operatorname{span}\{u\otimes v\mid u\ \text{a nonzero path with first arrow}\ \delta,\
v\in K_\delta\},
\end{equation*}
the sum being direct because distinct paths are linearly independent. Each $U_\delta$ is a
submodule of $G$: if $a$ is an arrow and $u$ a nonzero path with first arrow $\delta$, then $au$ is
zero or a nonzero path with first arrow $\delta$. Since the nonzero paths with first arrow
$\delta$ form a basis of $\La\delta\subseteq P_{s(\delta)}$, we have
$U_\delta=\La\delta\sh{-d}\otimes K_\delta$ as submodules of $P_{s(\delta)}\sh{-d}\otimes
W_d(s(\delta))$. This proves the decomposition.

Since $W$ is generated by $W_d$, the map $\pi$ is surjective, and $\Ker\pi$ lies in degrees
$\ge d+1$, hence in $\rad G$; so $\pi$ is a projective cover. Moreover $\Ker\pi$ is a submodule of
an object of $\La\proj$ which is generated in degree $d+1$ and vanishes in degrees $\le d$.
Applying the first part to $\Ker\pi$, and iterating, we obtain a minimal graded projective
resolution of $W$ whose $n$-th term is generated in degree $d+n$; that is, $W\sh d$ is Koszul.
\end{proof}

We now describe the second syzygies of finite-dimensional graded $\La$-modules. The
ungraded statement is due to Zimmermann Huisgen \cite{ZH}.

\begin{proposition}\label{prop:mon-second-syzygy}
For every $M\in\La\gmod$ there are arrows $\alpha_1,\dots,\alpha_s$ and integers
$i_1,\dots,i_s$ with $\Omega^2M\cong\bigoplus_{j}\La\alpha_j\sh{i_j}$. The same holds for
ungraded modules, without shifts.
\end{proposition}

\begin{proof}
By \cite{ZH}, the second syzygy of a finite-dimensional module over a monomial algebra is a
finite direct sum of modules $\La p$ with $p$ a path of positive length. By \eqref{eq:Lp},
$\La p\cong\La\lambda$ as ungraded modules, where $\lambda$ is the last arrow of $p$. This proves
the ungraded statement. The modules $\La\alpha$ are indecomposable, being cyclic with simple top.

Let now $M\in\La\gmod$, and let $\Omega^2M=\bigoplus_tN_t$ be a decomposition into indecomposable
graded modules. We use two results of Gordon and Green \cite{GG}: the underlying ungraded module
of an indecomposable graded module is indecomposable, and two indecomposable graded modules whose
underlying ungraded modules are isomorphic differ by a shift. By the first result and the
Krull--Schmidt theorem in $\La\fmod$, the ungraded module underlying each $N_t$ is isomorphic to
$\La\lambda_t$ for some arrow $\lambda_t$. Since $\La\lambda_t$ is a graded module, the second
result gives $N_t\cong\La\lambda_t\sh{i_t}$ for some $i_t\in\mathbb Z$.
\end{proof}

We are now able to show that quadratic monomial algebras are absolutely Koszul.

\begin{theorem}\label{thm:mon-abs-Koszul}
Let $\La$ be a finite-dimensional quadratic monomial algebra.
\begin{enumerate}[label=\textup{(\alph*)}]
\item $\La$ is absolutely Koszul and $\glld\La\le2$.
\item $\Ld$ is graded left co-coherent and $\Ld\Fcp=\Ld\Cop$.
\item $\Ld$ is graded left coherent, and every finitely presented graded $\Ld$-module has a
finite resolution by finitely generated graded projective modules.
\end{enumerate}
\end{theorem}

\begin{proof}
(a) By Proposition~\ref{prop:mon-second-syzygy}, $\Omega^2M$ is a direct sum of shifts of
modules $\La\alpha$, each of which has a linear resolution up to shift by
Proposition~\ref{prop:K-alphaI}. Hence $\Omega^2M$ is weakly Koszul, and $\ld(M)\le2$ by
Lemma~\ref{lem:syzygy}. (b) follows from (a) and Theorem~\ref{thm:A}. (c) The algebra
$\La^{\mathrm{op}}=kQ^{\mathrm{op}}/(R^{\mathrm{op}})$ is again finite-dimensional quadratic
monomial, and $(\La^{\mathrm{op}})^!\cong(\Ld)^{\mathrm{op}}$. The duality $\DD$ exchanges
finitely copresented (resp.\ coperfect) graded $(\Ld)^{\mathrm{op}}$-modules with finitely
presented (resp.\ perfect) graded $\Ld$-modules, so (c) follows from (b) applied to
$\La^{\mathrm{op}}$.
\end{proof}

For quadratic monomial algebras, the condition of Theorem~\ref{thm:C} is always
satisfied.

\begin{corollary}\label{cor:mon-linear-pres}
Let $\La$ be a finite-dimensional quadratic monomial algebra. Then
$\mathcal L(\La)=\mathcal K(\La)$. Consequently, every finite-dimensional graded $\Ld$-module
$N$ concentrated in degrees $-1$ and $0$ with $\soc N=N_0$ is isomorphic to $C_{\ge-1}$ for
some coKoszul $\Ld$-module $C$.
\end{corollary}

\begin{proof}
Let $0\neq M\in\mathcal L(\La)$. Its minimal presentation $P^{-1}\to P^0\to M\to0$ is
linear, so $\Omega M\subseteq\rad P^0$ is generated in degree $1$ and vanishes in degrees
$\le0$. By Lemma~\ref{lem:single-degree}, $\Omega M\sh1$ is Koszul, so $P_M$ is linear. The
second assertion follows from Theorem~\ref{thm:C}.
\end{proof}

\subsection{Finitely copresented \texorpdfstring{$\Ld$}{Lambda!}-modules}

For $M\in\Ld\GMod^-$ and $r\in\mathbb Z$, the subspace $M_{>r}=\bigoplus_{n>r}M_n$ is a
finite-dimensional submodule, and we put $M_{\le r}=M/M_{>r}$. As in the proof of
Theorem~\ref{thm:t-singular}(b), $K(M_{\le r})=\sigma^{\le r}K(M)$.

The finitely copresented graded $\Ld$-modules can now be described explicitly.

\begin{theorem}\label{thm:Fcp-characterization}
Let $\La$ be a finite-dimensional quadratic monomial algebra and $M\in\Ld\GMod$. Then
$M\in\Ld\Fcp$ if and only if there is a short exact sequence
\begin{equation*}
0\longrightarrow F\longrightarrow M\longrightarrow\bigoplus_{j=1}^{s}\Ia{\alpha_j}\sh i\longrightarrow0
\end{equation*}
with $F$ finite-dimensional, $\alpha_j\in Q_1$ not necessarily distinct, and $i\in\mathbb Z$.
More precisely, if $M\in\Ld\Fcp$ and $H^n(K(M))=0$ for all $n\le r+1$, then
$M_{\le r}\cong\bigoplus_j\Ia{\alpha_j}\sh{-1-r}$, and one may take $F=M_{>r}$.
\end{theorem}

\begin{proof}
If such a sequence exists, then $F$ and the modules $\Ia{\alpha_j}\sh i$ are coperfect
(Proposition~\ref{prop:K-alphaI}), hence so is $M$.

Conversely, let $M\in\Ld\Fcp=\Ld\Cop$ and choose $r$ with $H^n(K(M))=0$ for $n\le r+1$, which
is possible since $K(M)$ has bounded cohomology. Then $\sigma^{\le r}K(M)=K(M_{\le r})$ is a
linear resolution of $C:=\Coker(d^{r-1})$, placed in degree $r$. By exactness at $r$ and
$r+1$, $C\cong\Img d^r=\Ker d^{r+1}$. The module $W=\Img d^{r+1}$ is generated in the single
degree $-r-1$ and vanishes below it, and $K(M)^{r+1}\to W$ is a projective cover, since its
kernel $\Img d^r$ lies in the radical. By Lemma~\ref{lem:single-degree},
$C\cong\bigoplus_j\La\alpha_j\sh{r+1}$ for suitable arrows $\alpha_j$. Put
$T=\bigoplus_j\Ia{\alpha_j}\sh{-1-r}$. By Proposition~\ref{prop:K-alphaI} and
\eqref{eq:K-shift}, $K(T)$ is a linear complex vanishing in degrees $>r$, exact in degrees
$<r$, with $H^r(K(T))\cong C$. Since minimal resolutions are unique up to isomorphism,
$K(M_{\le r})\cong K(T)$, hence $M_{\le r}\cong T$ because $K$ is fully faithful.
\end{proof}

\subsection{The cotails category}

Put $\mathcal A:=\Ld\Fcp/\Ld\gmod$. We first describe the objects of $\mathcal A$.

\begin{corollary}\label{cor:objects-cotails}
Every object of $\mathcal A$ is isomorphic to the image of a finite direct sum
$\bigoplus_j\Ia{\alpha_j}\sh i$ with a common shift $i$. For every arrow $\alpha\colon x\to y$
there is an isomorphism $\Ia{\alpha}\cong\bigoplus_{\gamma\alpha\in R}\Ia{\gamma}\sh1$ in
$\mathcal A$, so the common shift may be replaced by any $i'\ge i$.
\end{corollary}

\begin{proof}
This follows from Theorem~\ref{thm:Fcp-characterization} and Lemma~\ref{lem:alphaI}(d),
since finite-dimensional modules become zero in $\mathcal A$.
\end{proof}

Next we show that $\mathcal A$ is hereditary.

\begin{corollary}\label{cor:hereditary}
The category $\mathcal A$ is hereditary, and every object of $\Db(\mathcal A)$ is isomorphic to
a finite direct sum of shifts of objects of $\mathcal A$. The same holds for
$\Ld\Fp/\Ld\gmod$.
\end{corollary}

\begin{proof}
Let $\Phi\colon\Db(\mathcal A)\to\Dsg(\La\gmod)$ be as in Subsection~\ref{sec:t-structures}.
Every object of $\Dsg(\La\gmod)$ is isomorphic to $N[m]\cong\Omega^2N[m+2]$ for some
$N\in\La\gmod$, hence, by Proposition~\ref{prop:mon-second-syzygy}, to a finite direct sum of
objects $\La\alpha\sh i[m]$. By Proposition~\ref{prop:K-alphaI} and \eqref{eq:K-shift},
\begin{equation}\label{eq:Phi-alpha}
\Phi\big(\Ia{\alpha}\sh i\big)\cong\pi\big(\La\alpha\sh{-i}[1+i]\big),
\end{equation}
so every object of $\Db(\mathcal A)$ is a finite direct sum of shifted stalk complexes.

Let $X,Y\in\mathcal A$ and $\xi\in\Hom_{\Db(\mathcal A)}(X,Y[2])$, with cone $C$, so that
$X\xrightarrow{\xi}Y[2]\xrightarrow{v}C\to X[1]$. Then $C\cong Y[2]\oplus X[1]$ by the first
part, since $H^{-2}(C)\cong Y$, $H^{-1}(C)\cong X$ and $C$ has no other cohomology. Since
$\Hom(Y[2],X[1])=0$, the map $v$ has the form $(v_1,0)$, with $v_1$ an isomorphism because
$H^{-2}(v)$ is. Hence $v$ is a split monomorphism and $\xi=0$. The last assertion follows by
applying $\DD$ and the first part to $\La^{\mathrm{op}}$.
\end{proof}

We use that, for $M,N\in\La\gmod$, there is a natural isomorphism
\begin{equation}\label{eq:Dsg-colim}
\Hom_{\Dsg(\La\gmod)}(M,N)\cong\varinjlim_{n}\uHom_{\La\gmod}\big(\Omega^nM,\Omega^nN\big),
\end{equation}
where $\uHom_{\La\gmod}$ denotes morphisms modulo those factoring through graded projective
modules, and the transition maps are induced by the syzygy functor $\Omega$ on the graded stable
category; see \cite{KV,Buc}. Recall that if $f\colon X\to Y$ is a morphism in $\La\gmod$ and
$g\colon P_X\to P_Y$ is a lift of $f$ along projective covers, then the restriction
$\Omega X\to\Omega Y$ of $g$ represents $\Omega(f)$ in the stable category, and its stable class
does not depend on the choice of $g$.

We also need the following vanishing result.

\begin{lemma}\label{lem:weights}
Let $\gamma,\beta\in Q_1$ and $a\neq b$. Then
$\Hom_{\Dsg(\La\gmod)}\big(\La\gamma\sh a,\La\beta\sh b\big)=0$.
\end{lemma}

\begin{proof}
By \eqref{eq:syzygy-Lalpha}, $\Omega^n(\La\gamma\sh a)$ and $\Omega^n(\La\beta\sh b)$ are direct
sums of modules $\La\delta\sh{a-n}$ and $\La\varepsilon\sh{b-n}$ respectively, with
$\delta,\varepsilon\in Q_1$. An element of the left-hand side of \eqref{eq:Dsg-colim}, with
$M=\La\gamma\sh a$ and $N=\La\beta\sh b$, is represented by a stable morphism
$h\colon\Omega^nM\to\Omega^nN$ for some $n$, and it vanishes as soon as $\Omega(h)=0$ in the
stable category. Computing $\Omega(h)$ componentwise, it is therefore enough to show the
following: for arrows $\delta,\varepsilon$ and integers $c_1\neq c_2$, every morphism
$f\colon\La\delta\sh{c_1}\to\La\varepsilon\sh{c_2}$ admits a lift along projective covers whose
restriction to $\Omega(\La\delta\sh{c_1})$ is zero.

The module $\La\delta\sh{c_1}$ is generated in degree $1-c_1$, and $\La\varepsilon\sh{c_2}$
vanishes in degrees $<1-c_2$. Hence $f=0$ if $c_1>c_2$, and we may assume that $c_1<c_2$. Put
$m=c_2-c_1\ge1$. Then $f(\delta)$ lies in $\La\varepsilon$ in degree $1+m$, so
$f(\delta)=\sum_uc_u\,u\varepsilon$, where $u$ runs through the paths of length $m$ from
$t(\varepsilon)$ to $t(\delta)$ with $u\varepsilon\neq0$ and $c_u\in k$. By \eqref{eq:monomial-product}, the first arrow $\eta_u$ of
each such $u$ satisfies $\eta_u\varepsilon\notin R$. The map
$g\colon P_{t(\delta)}\sh{c_1-1}\to P_{t(\varepsilon)}\sh{c_2-1}$, $e_{t(\delta)}\mapsto\sum_uc_uu$, lifts
$f$ along the projective covers $e_{t(\delta)}\mapsto\delta$ and $e_{t(\varepsilon)}\mapsto\varepsilon$,
so $\Omega(f)$ is represented by the restriction of $g$ to
$\Omega(\La\delta\sh{c_1})=\bigoplus_{\beta'\delta\in R}\La\beta'\sh{c_1-1}$. For an arrow $\beta'$ with
$\beta'\delta\in R$ we have $g(\beta')=\sum_uc_u\,\beta'u$, and each nonzero path $\beta'u$ has first
arrow $\eta_u$. On the other hand, $g(\beta')$ lies in
$\Omega(\La\varepsilon\sh{c_2})=\bigoplus_{\theta\varepsilon\in R}\La\theta\sh{c_2-1}$, which is spanned
by paths whose first arrow $\theta$ satisfies $\theta\varepsilon\in R$. Since distinct paths are
linearly independent, $g(\beta')=0$. Hence the restriction of $g$ to $\Omega(\La\delta\sh{c_1})$
is zero, and $\Omega(f)=0$ in the stable category.
\end{proof}

In fact, more is true.

\begin{corollary}\label{cor:semisimple}
Let $\La$ be a finite-dimensional quadratic monomial algebra. Then $\Ld\Fcp/\Ld\gmod$ and
$\Ld\Fp/\Ld\gmod$ are semisimple abelian categories, that is, every short exact sequence in them
splits.
\end{corollary}

\begin{proof}
By Corollary~\ref{cor:objects-cotails} it suffices to show that
$\Ext^1_{\mathcal A}(\Ia{\alpha}\sh i,\Ia{\beta}\sh i)=0$. By \eqref{eq:Phi-alpha}, this group
is isomorphic to $\Hom_{\Dsg(\La\gmod)}(\La\alpha,\La\beta[1])$. Since
$\La\alpha\cong\Omega(\La\alpha)[1]$ in $\Dsg(\La\gmod)$, it is isomorphic to
$\bigoplus_{\gamma\alpha\in R}\Hom_{\Dsg(\La\gmod)}(\La\gamma\sh{-1},\La\beta)$ by
\eqref{eq:syzygy-Lalpha}, which vanishes by Lemma~\ref{lem:weights}. The category
$\Ld\Fp/\Ld\gmod$ is treated by duality, as in the proof of Corollary~\ref{cor:hereditary}.
\end{proof}

\subsection{The singular Koszul \texorpdfstring{$t$}{t}-structure}

For quadratic monomial algebras, the $t$-structure of Theorem~\ref{thm:t-singular} can be
described as follows.

\begin{theorem}\label{thm:t-monomial}
Let $\La$ be a finite-dimensional quadratic monomial algebra.
\begin{enumerate}[label=\textup{(\alph*)}]
\item Every object of $\Dsg(\La\gmod)$ is isomorphic to a finite direct sum
$\bigoplus_n\Phi(X_n)[-n]$ with $X_n\in\mathcal A$, and it lies in $\mathsf D^{\le0}_{\Phi}$
(resp.\ $\mathsf D^{\ge0}_{\Phi}$) if and only if $X_n=0$ for $n>0$ (resp.\ $n<0$).
\item The heart $\mathcal H_\Phi\simeq\mathcal A$ is hereditary, and its objects are, up to
isomorphism, the objects $\pi(L[g])$ with $g\in\mathbb Z$ and
$L\cong\bigoplus_j\La\alpha_j\sh{1-g}$.
\item The heart $\mathcal H_\Phi$ is semisimple.
\end{enumerate}
\end{theorem}

\begin{proof}
(a) follows from Corollary~\ref{cor:hereditary}, since an object
$\bigoplus_nX_n[-n]$ of $\Db(\mathcal A)$ has cohomology $X_n$ in degree $n$. (b) By
Corollary~\ref{cor:objects-cotails} and \eqref{eq:Phi-alpha}, the objects of
$\mathcal H_\Phi$ are the images $\pi\big(\bigoplus_j\La\alpha_j\sh{-i}[1+i]\big)$; put $g=1+i$.
(c) is Corollary~\ref{cor:semisimple}.
\end{proof}

\subsection{The finitistic dimension}

For an arrow $\alpha$ of $Q$ put
\begin{equation*}
L_\alpha:=\sup\{\ell(q)\mid q=\alpha^{\mathrm{op}}q'\ \text{is a nonzero path of }\Ld\}\in\mathbb N\cup\{\infty\},
\end{equation*}
the supremum of the lengths of the $R$-chains starting with $\alpha$, by \eqref{eq:R-chains}.

\begin{theorem}\label{thm:mon-findim}
Let $\La$ be a finite-dimensional quadratic monomial algebra.
\begin{enumerate}[label=\textup{(\alph*)}]
\item $\pd\La\alpha=L_\alpha-1$ for every arrow $\alpha$; in particular $\pd\La\alpha<\infty$ if
and only if $\Ia{\alpha}$ is finite-dimensional.
\item $\findim\La\le1+\max\{L_\alpha\mid\alpha\in Q_1,\ L_\alpha<\infty\}$, where the maximum
of the empty set is $0$.
\end{enumerate}
\end{theorem}

\begin{proof}
(a) By Proposition~\ref{prop:K-alphaI}, $K(\Ia{\alpha})$ is the minimal resolution of
$\La\alpha$ placed in degree $-1$, and its term in degree $-1-n$ is nonzero if and only if
there is a nonzero path $\alpha^{\mathrm{op}}q'$ of $\Ld$ with $\ell(q')=n$. (b) Let $M\in\La\fmod$ with $\pd M<\infty$. If $\Omega^2M=0$, then
$\pd M\le1$. Otherwise $\Omega^2M\cong\bigoplus_j\La\alpha_j$ by
Proposition~\ref{prop:mon-second-syzygy}, each $L_{\alpha_j}$ is finite, and
$\pd M=2+\max_j(L_{\alpha_j}-1)$ by (a).
\end{proof}

We can now compare the finitistic dimension with its graded version.

\begin{corollary}\label{cor:findim-grfindim}
For a finite-dimensional quadratic monomial algebra $\La$,
$\findim\La\le\grfindim\La+2$. In particular, Conjecture~\ref{conj:findim} holds for $\La$.
\end{corollary}

\begin{proof}
The modules $\La\alpha$ are graded, so $\grfindim\La\ge L_\alpha-1$ whenever $L_\alpha<\infty$,
by Theorem~\ref{thm:mon-findim}(a). Since also $\grfindim\La\ge0$, the claim follows from
Theorem~\ref{thm:mon-findim}(b).
\end{proof}


\end{document}